\documentclass[11pt,reqno]{amsart}

\usepackage{amsmath,amssymb,mathrsfs}
\usepackage{graphicx,cite,times}
\usepackage{cases}
\usepackage{color}

\newtheorem{theo}{Theorem}[section]

\newtheorem{lemm}[theo]{Lemma}

\newtheorem{rema}[theo]{Remark}

\newtheorem{assu}{Hypothesis}
\numberwithin{equation}{section}

\begin{document}

\title[Periodic Solutions to nonlinear Euler-Bernoulli beam equations]{
Periodic Solutions to nonlinear Euler-Bernoulli beam equations}

\author{Bochao Chen}
\address{School of
Mathematics and Statistics, Center for Mathematics and
Interdisciplinary Sciences, Northeast Normal University, Changchun, Jilin 130024, P.R.China}
\email{chenbc758@nenu.edu.cn}

\author{Yixian Gao}
\address{School of
Mathematics and Statistics, Center for Mathematics and
Interdisciplinary Sciences, Northeast Normal University, Changchun, Jilin 130024, P.R.China}
\email{gaoyx643@nenu.edu.cn}

\author{Yong Li}
\address{School of
Mathematics and Statistics, Center for Mathematics and
Interdisciplinary Sciences, Northeast Normal University, Changchun, Jilin 130024, P.R.China. }
\email{yongli@nenu.edu.cn}
%
%
\thanks{The research of YG was supported in part by  FRFCU2412017FZ005
and JLSTDP 20160520094JH.
 The research  of YL was supported in part by NSFC grant 11571065  and  National Research Program of China Grant 2013CB834100}
 


\begin{abstract}
Bending vibrations of thin beams and plates may be described by nonlinear Euler-Bernoulli beam equations with  $x$-dependent  coefficients.
In this paper we investigate existence of families of time-periodic solutions to such a model using  Lyapunov-Schmidt reduction and
 a differentiable  Nash-Moser iteration scheme. The results hold for all  parameters $(\epsilon,\omega)$ in a Cantor set with asymptotically full measure as $\epsilon\rightarrow0$.
\end{abstract}

\maketitle

\section{Introduction}

Consider  one dimensional (1D)
 nonlinear Euler-Bernoulli beam equations
\begin{align}\label{A1}
\rho(x)u_{tt}+(p(x)u_{xx})_{xx}=\epsilon f(\omega t,x,u),\quad x\in[0,\pi]
\end{align}
with respect to the pinned-pinned boundary conditions:
\begin{align}\label{A2}
u(t,0)=u(t,\pi)=u_{xx}(t,0)=u_{xx}(t,\pi)=0,
\end{align}
where $\rho, p$ are positive coefficients, the parameter $\epsilon$ is small, and the nonlinear forcing term $f(\omega t,x, u)$ is $\frac{2\pi}{\omega}$-periodic in time, i.e., $f(\cdot,x, u)$ is $2\pi$-periodic. Obviously, $u=0$ is not the solution of equation \eqref{A1} if $f(\omega t,x,0)\neq0$.

Bending vibrations of thin beams and plates may be described by equation \eqref{A1}, which reflects the relationship between  the applied load and the beam's deflection, see \cite{weaver1990vibration}. The curve $u(\cdot,x)$ describes the deflection of the beam at some position $x$ in the vertical direction, $p$  is the flexural rigidity and $\rho$ is the density of the beam. Derivatives of the deflection $u$ have  physical significance: $u_x$ is the slope of the beam; $-p u_{xx}$ is the bending moment of the beam and $-(pu_{xx})_{x}$ is the shear force of the beam. Moreover $f$ is distributed load,  which may be a function of $x$, $u$ or other variables.

The free vibration of uniform and non-uniform beams attracted many investigators since Bernoulli and
Euler derived the governing differential equation in the 18th century. The beams with end springs have been
dealt with by many investigators. Many researchers focused on the study of the spectral problems for the following Euler-Bernoulli operators
\begin{align}\label{A5}
\mathcal{E}u:=\frac{1}{\rho}(pu'')''+Vu,
\end{align}
see \cite{MR1334400,Badanin2017resonances,MR1660270,MR2594376,MR1990171,MR3346145} and references therein. Elishakoff et al. considered apparently the first time harmonic form solution (i.e. $u(t,x)=u(x)\sin \omega t$) for linear  equation of \eqref{A1} under different boundary conditions, see \cite{MR2267999,MR2160080}.  Under linear boundary feedback control, in \cite{MR1897202}, Guo was concerned with  the Riesz basis property and the stability of such one with boundary conditions
\begin{align*}
\begin{cases}
y(t, 0)=y_x(t, 0)=y_{xx}(t, \pi)=0,\\
(p(x)y_{xx})_{x}(t, \pi ) = ky_t( t, \pi),
\end{cases}
\end{align*}
where $k \geq0$ is a constant feedback.  Despite many studies on the linear model above, the nonlinear problems are less studied due to
 the challenge of the invertibility of linearized Euler-Bernoulli operators with
   variable coefficients $\rho,p$. This paper presents the first mathematical  analysis for
the existence of periodic solutions  to nonlinear  equation \eqref{A1}. There are two main challenges in this work: (i) The finite differentiable  regularities of the nonlinearity. Clearly, a difficulty when working with functions
having only Sobolev regularity is that the Green functions will exhibit only a polynomial decay
off the diagonal, and not exponential (or subexponential). A key concept that one must exploit
is the interpolation/tame estimates.
(ii) The  ``small divisors problem'' caused by resonances. We  give  the  asymptotic formulae of the eigenvalues to  the Euler-Bernoulli beam's problem \eqref{D4}. The asymptotic property of the eigenvalues for fourth-order  operators on the unit interval are less investigated than for second-order ones, see also \cite{najmark1968linear,Caudill1998isospectral}.

 Letting  $t\rightarrow{t}/{\omega}$, equation \eqref{A1} is equivalent to
\begin{align}\label{B1}
\omega^2\rho(x)u_{tt}+(p(x)u_{xx})_{xx}=\epsilon f(t,x,u)
\end{align}
Hence we look for $2\pi$-periodic solutions in time to \eqref{B1}. The existence problem of periodic or quasi-periodic solutions for PDEs has received considerable attention in the last twenty years. The main difficulty in finding periodic solutions of \eqref{B1} is the so-called ``small divisors problem'' caused by resonances. In fact, the spectrum of
$${\mathcal{M}}u:=\omega^2u_{tt}+\frac{1}{\rho}(pu_{xx})_{xx}$$
presents the  following form
\begin{equation*}
-\omega^2 l^2+\lambda_j=-\omega^2{l^2}+j^4+aj^2+b+O({1}/{j}), \quad l\in\mathbb{Z},~j\rightarrow +\infty.
\end{equation*}
Consequently,   above spectrum approaches to zero for almost every $\omega$ under the assumption  $b\neq0$.
 This causes that the operator ${\mathcal{M}}$ cannot map, in general, a  functional space into itself, but only into  a large  functional space with less regularity. There are two main approaches to deal with ``small divisors problem''. One is the infinite-dimensional KAM (Kolmogorov-Arnold-Moser) theory to  Hamiltonian PDEs, refer to Kuksin \cite{MR911772}, Wayne \cite{MR1040892}, and recent results \cite{MR3668616,MR3187681,MR2507958}. The other more direct bifurcation approach was
 established by  Craig
 and Wayne \cite{MR1239318} and improved by Bourgain \cite{MR1316975,MR1345016} based on
 Lyapunov-Schmidt reduction and a Nash-Moser iteration
  procedure, and recent results \cite{MR3312439,MR2831876,MR2580515}.

Up to now,  there has been  a number of work devoted to the existence of periodic solutions  and  quasi-periodic solutions for  classical beam  equation, i.e. equation \eqref{A1} with $\rho(x)=p(x)=1$.
In \cite{Mckenna1999,Mckenna1987}, Mckenna et al. investigated the nonlinear beam equation as  a  model for  a   suspension bridge and established  multiple periodic solutions when a parameter exceeds a certain eigenvalue.
By means of the infinite KAM theorem,  the existence and stability  of  small-amplitude quasi-periodic solutions of one dimensional beam equations with boundary conditions \eqref{A2} was obtained in
\cite{chang2015quasi,geng2003kam,wang2012quasi}. For high dimensional cases,
in \cite{eliasson2016kam}, Eliasson, Gr\'{e}bert and Kuksin proved that there exist many  linearly stable or unstable (for $d \geq 2$) small-amplitude quasi-periodic solutions for nonlinear beam equations
\begin{equation*}
u_{tt}+\Delta^{2}u+mu+\partial_{u}f(x,u)=0, \quad x\in\mathbb{T}^{d},
\end{equation*}
where $f(x,u)=u^4+O(u^5)$. Above proofs are carried out in analytic cases.
 Recently, based on a differentiable  Nash-Moser type implicit function theorem, in \cite{chen2018quasi}, Chen et al. gave
 the  existence
of quasi-periodic in time solutions of nonlinear beam equations
\begin{align*}
u_{tt}+\Delta^{2}u+V(x)u=\epsilon f(\omega t, x, u),\quad  x\in\boldsymbol M,
\end{align*}
where $\boldsymbol M$ is any compact Lie group or  homogenous manifold with respect to a compact Lie group.
The nonlinear  PDEs with $x$-dependent coefficients   have recently attracted considerable
attention due to their widely application. In \cite{Barbu1996,Barbu1997}, Barbu and Pavel considered the wave equations
with $x$-dependent coefficients for the first time.
Under the general boundary conditions, periodic or anti-periodic boundary conditions and Dirichle boundary conditions, Ji and Li showed the existence of  periodic solutions which periodic $T$ is required to be a rational multiple $\pi$, see \cite{ji2007periodic,ji2018periodic,ji2011periodic}.
The existence of  periodic  solutions with  $T$ being  a irrational multiple $\pi$   for the  forced vibrations of a nonhomogeneous string  was obtained in \cite{baldi2008forced} and \cite{Chen2018} by a Nash-Moser theorem.

\subsection{Main results}

We now state the main results of this paper. To do so, we need to make our notations and assumptions more precise.

Let $\rho,p$ satisfy
 \begin{align}\label{A3}
 \rho(x)=e^{4\int^{x}_{0}\alpha(\mathfrak{x}) \mathrm{d}\mathfrak{x}}>0,\quad p(x)=p(0)e^{4\int^{x}_0\beta(\mathfrak{x})\mathrm{d}\mathfrak{x}}>0
 \end{align}
 with $\alpha(0)+\beta(0)=\alpha(\pi)+\beta(\pi)=0.$
Without loss of generality we assume  the following normalization:
\begin{align*}
\int^{\pi}_{0}\left({\rho/p}\right)^{\frac{1}{4}}~\mathrm{d}x=\pi.
\end{align*}

In fact, assumption \eqref{A3} on $\rho,p$ is essential for  giving the asymptotic forms of the spectrum  of  Euler-Bernoulli beam operator.
 Make the  Liouville substitution
\begin{align*}
x=\psi(\xi)\Leftrightarrow \xi=\phi(x)\quad\text{with}~ \phi(x):=\int^{x}_{0}\zeta(s)\mathrm{d}s,~ \zeta(s)=\left({\rho}/{p}\right)^{1/4},
\end{align*}
together with the unitary Barcilon-Gottlieb transformation $\mathrm{U}:L^2((0,\pi),\rho(x)\mathrm{d}x)\rightarrow L^2((0,\pi),\mathrm{d}\xi)$ by
\begin{align*}
u(x)\mapsto y(\xi)=(\mathrm{U}u)(\psi(\xi))=q(\psi(\xi))u(\psi(\xi)),\quad \xi\in[0,\pi],
\end{align*}
where $q=p^{1/8}\rho^{3/8}>0$.
Under the conditions that  $\rho,p$ satisfy special condition \eqref{A3}, in  Lemma 5.1 of
 \cite{MR3346145}, Badanin and ~Korotyaev  showed that
the operators $\mathcal{E}$ defined on \eqref{A5} satisfying boundary conditions \eqref{A2} and $\mathcal{H}$ satisfying boundary conditions \eqref{A2} are unitarily equivalent and one has $\mathcal{E}=\mathrm{U}^{-1}\mathcal{H}\mathrm{U}$, where
\begin{align*}
\mathcal{H}u:=u_{xxxx}+2(p_1u_x)_x+p_2u,
\end{align*}
and $p_i,i=1,2$ are seen in (5.4)--(5.5) of \cite{MR3346145}. Then, They  applied  the asymptotic formulae of the eigenvalues of the operators $\mathcal{H}$ to study the ones of the operators $\mathcal{E}$, see \cite{MR3346145}.

For all $s\geq0$, define the following Sobolev spaces $\mathcal{H}^{s}$ of real-valued functions by
\begin{align*}
\mathcal{H}^{s}:=\left\{u:\mathbb{T}\rightarrow H^2_p((0, \pi);\mathbb{R}),u(t,x)=\sum\limits_{l\in\mathbb{Z}}u_{l}(x)e^{{\rm i}lt},u_l\in{H}^2_p((0, \pi);\mathbb{C}),~ u_{-l}=u^*_{l},~\|u\|_{s}<+\infty \right\},
\end{align*}
where $u^*_{l}$ is the complex conjugate of $u_l$, $\|u\|^{2}_{s}:=\sum\limits_{l\in\mathbb{Z}}\|u_l\|^2_{H^2}(1+ l^{2s})$ and
\begin{align*}
{H}^2_p((0,\pi);\mathbb{C}):=&\bigg\{u\in H^2((0,\pi);\mathbb{C}):~u(t,0)=u(t,\pi)=u_{xx}(t,0)=u_{xx}(t,\pi)=0\bigg\}.
\end{align*}
For all $s>\frac{1}{2}$, one has the classical Sobolev embedding $\mathcal{H}^s\hookrightarrow L^{\infty}(\mathbb{T};{H}^2_{p}(0,\pi))$ satisfying
\begin{align}\label{B2}
\quad\|u\|_{L^{\infty}(\mathbb{T};{H}^2_p(0,\pi))}\leq C(s)\|u\|_{s},\quad\forall u\in \mathcal{H}^s,
\end{align}
and algebra property, i.e.,
\begin{align*}
\|uv\|_{{s}}\leq C(s)\|u\|_{s}\|v\|_{s}, \quad\forall u,v\in\mathcal{{H}}^s.
\end{align*}
Throughout this paper, our purpose is to look for the  solutions  in $\mathcal{H}^s$ with respect to $(t, x)\in \mathbb{T}\times[0,\pi]$ and $f\in\mathcal{C}_k$ for $k\in\mathbb{N}$ large enough, where
\begin{align*}
\mathcal{C}_k:=\left\{f\in C^1(\mathbb{T}\times [0,\pi]\times\mathbb{R};\mathbb{R}):(t,u)\mapsto f(t,\cdot,u)\text{ belongs to } C^k\left(\mathbb{T}\times\mathbb{R};H^2(0,\pi)\right)\right\}.
\end{align*}
\begin{rema}
If $f(t,x,u)=\sum_{l\in\mathbb{Z}}f_{l}(x,u)e^{{\rm i}l t}$, then $u\mapsto f_{l}(\cdot,u)\in C^{k}(\mathbb{R};H^{2}((0,\pi);\mathbb{C}))$ with $f_{-l}={f}^*_{l}$. Moreover it follows from  the  continuously embedding  of  ${H}^2(0,\pi)$ into   $C^1[0,\pi]$  that
\begin{align*}
\partial_{t}^{i}\partial_{u}^{j}f\in C^1(\mathbb{T}\times[0,\pi]\times\mathbb{R};\mathbb{R}),\quad\forall0\leq i,j\leq k, \forall f\in\mathcal{C}_k.
\end{align*}
\end{rema}

Denoting by
\begin{align*}
V:={H}^2_p(0,\pi),\quad W:=\left\{ w=\sum\limits_{l\in\mathbb{Z},l\neq0}w_{l}(x)e^{\mathrm{i}l t}\in \mathcal{H}^0\right\},
\end{align*}
we perform the Lyapunov-Schmidt reduction subject to the following decomposition
\begin{align*}
\mathcal{H}^s=(V\cap \mathcal{H}^s)\oplus (W\cap \mathcal{H}^s)=V\oplus (W\cap \mathcal{H}^s).
\end{align*}
In fact, for any  $u\in \mathcal{H}^s$, one has $u(t,x)=u_0(x)+\bar{u}(t,x)$ with $\bar{u}(t,x)=\sum_{l\neq0}u_{l}(x)e^{\mathrm{i}lt}$. Then corresponding projectors
$\Pi_{V}: \mathcal{H}^s\rightarrow V,~\Pi_{W}:\mathcal{H}^s\rightarrow W$ yields that equation \eqref{B1} is equivalent to
\begin{align}\label{C1}
\begin{cases}
(pv'')''=\epsilon\Pi_{V}F(v+w)\quad &(Q),\\
L_{\omega}w=\epsilon\Pi_{W}F(v+w)\quad &(P),
\end{cases}
\end{align}
where $u=v+w$ with $v\in V$, $w\in W$, and
\begin{align}\label{C2}
L_{\omega}w:=\omega^2\rho(x)w_{tt}+(p(x)w_{xx})_{xx},\quad F:u\rightarrow f(t,x,u).
\end{align}
Note that equations $(Q)$ and $(P)$ are called bifurcation equation and  range equation, respectively. Moreover we may write $f$ as
\begin{align*}
f(t,x,u)=f_{0}(x, u)+\bar{f}(t,x,u),
\end{align*}
where $\bar{f}(t,x,u)=\sum_{l\neq0}f_{l}(x,u)e^{{\rm i}lt}$. This  leads to
 \begin{align*}
\Pi_{V}F(v)=\Pi_{V}f(t,x,v(x))=\Pi_{V}f_{0}(x, v(x))+\Pi_{V}\bar{f}(t,x,v(x))=f_{0}(x, v(x))\quad \text{for }w=0.
\end{align*}
If $w$ tends to 0, then we simplify the $(Q)$-equation as
\begin{align*}
(pv'')''=\epsilon f_{0}(x, v),
\end{align*}
which is also called the infinite-dimensional ``zeroth-order bifurcation equation'', see also \cite{baldi2008forced}. We need make the following hypothesis.
\begin{assu}\label{hy1}
There exists a constant $\epsilon_0\in(0,1)$ small enough, such that for all $\epsilon\in[0,\epsilon_0]$, the following system
\begin{align}\label{C3}
\begin{cases}
(p(x)v''(x))''=\epsilon f_{0}(x,v(x)),\\
v(0)=v(\pi)=v''(0)=v''(\pi)=0
\end{cases}
\end{align}
admits a nondegenerate solution $\hat{{v}}\in {H}^2_p(0,\pi)$, i.e., the linearized equation
\begin{equation}\label{C4}
(ph'')''=\epsilon f'_0(\hat{{v}})h
\end{equation}
possesses only the trivial solution $h=0$ in ${H}^2_p(0,\pi)$.
\end{assu}
Let us explain the rationality of Hypothesis \ref{hy1}. The linearized equation \eqref{C4} possesses only the trivial solution $h=0$ in ${H}^2_p(0,\pi)$ for $\epsilon=0$. Hence $\hat{v}=0$ is the nondegenerate solution of \eqref{C3} with $\epsilon=0$. It follows from the implicit function theorem that there exists a constant $\epsilon_0 \in(0,1)$ small enough, such that for all $\epsilon\in[0,\epsilon_0]$, Hypothesis \ref{hy1} is satisfied. Moreover define
\begin{align*}
A_{\gamma}:=\left\{(\epsilon,\omega)\in(\epsilon_1,\epsilon_2)\times(\gamma,+\infty):\frac{\epsilon}{\omega}\leq\delta_7\gamma^5, \left|\omega l-\bar{\mu}_j\right|>\frac{\gamma}{l^{\tau}},
~\forall l=1,\cdots,N_0,~\forall j\geq1
\right\},
\end{align*}
where $\delta_7$ is given in Lemma \ref{lem18}, $N_0$ is seen in \eqref{D9} and $\bar{\lambda}_j=\bar{\mu}^2_j,j\geq1$ are the eigenvalues of Euler-Bernoulli beam's problem
\begin{align}\label{B3}
\begin{cases}
(p(x)y'')''=\lambda \rho(x)y,\\
y(0)=y(\pi)=y''(0)=y''(\pi)=0.
\end{cases}
\end{align}
Let us state our main theorem as follows.
\begin{theo}\label{Th1}
Assume  that Hypotheses \ref{hy1} holds for some $\hat{\epsilon}\in[0,{\epsilon}_0]$. Set $\alpha(x),\beta(x)\in H^4(0,\pi)$, $f\in \mathcal{C}_k$ for all $k\geq s+\kappa+3$, where
\begin{align}\label{C6}
\kappa:=6\tau+4\sigma+2
\end{align}
with $\sigma={\tau(\tau-1)}/{(2-\tau)}$, and fix $\tau\in(1,2),\gamma\in(0,1)$. Provided that  $\rho,p$ satisfy \eqref{A3}, if  $\frac{\epsilon}{\gamma^5\omega}\leq\delta_7$ is small enough, then there exist a constant $K>0$ depending on $\alpha,\beta, f, \epsilon_0, \hat{v}, \gamma, \gamma_0, \tau, s,\kappa$, a neighborhood $(\epsilon_1,\epsilon_2)$ of $\hat{\epsilon}$, $0<r<1$ and a $C^2$ map $v(\epsilon,w)$ defined on $(\epsilon_1,\epsilon_2)\times\left\{w\in W\cap \mathcal{H}^s:~\|w\|_{s}< r\right\}$ with values in $H^2_p(0,\pi)$ satisfying
\begin{align}\label{C7}
\| v(\epsilon,w)-v(\epsilon,0)\|_{H^2}\leq K\|w\|_s,\quad\|v(\epsilon,0)-\hat{v}\|_{H^2}\leq K|\epsilon-\hat{\epsilon}|,
\end{align}
a map $\tilde{w}\in C^1(A_\gamma;W\cap \mathcal{H}^s)$ satisfying
\begin{align}\label{C8}
\|\tilde{w}\|_{s}\leq\frac{K\epsilon}{\gamma\omega},
\quad\|\partial_{\omega}\tilde{w}\|_{s}\leq\frac{K\epsilon}{\gamma^5\omega},
\quad\|\partial_{\epsilon}\tilde{w}\|_{s}\leq\frac{K}{\gamma^5\omega},
\end{align}
 such that for all $(\epsilon,\omega)\in B_\gamma\subset A_\gamma$,
\begin{align*}
\tilde{u}:=v(\epsilon,\tilde{w}(\epsilon,\omega))+\tilde{w}(\epsilon,\omega)\in H^{6}{(0,\pi)}\cap H^{2}_p{(0,\pi)},\quad \forall t\in\mathbb{T}
\end{align*}
is a solution of equation \eqref{B1}, where $B_\gamma$ is defined in \eqref{G9}. Moreover the Lebesgue measures of the set $B_{\gamma}\subset A_\gamma$ and its section $B_{\gamma}(\epsilon)$ satisfy
\begin{align*}
\mathrm{meas}(B_{\gamma}(\epsilon)\cap(\omega',\omega''))\geq(1-K\gamma)(\omega''-\omega'),\quad
\mathrm{meas}(B_{\gamma}\cap\Omega)\geq(1-K\gamma)\mathrm{meas}(\Omega),
\end{align*}
where $\Omega:=(\epsilon',\epsilon'')\times(\omega',\omega'')$ stands for a rectangle contained in $(\epsilon_1,\epsilon_2)\times(2\gamma,+\infty)$.
\end{theo}

\subsection{Plan of the paper}
The rest of the paper is organized as follows.
In section \ref{sec:2.1}, we solve the $(Q)$-equation by the classical implicit function theorem under Hypothesis \ref{hy1}.
 The goal of section  \ref{sec:2.2}
 is  to solve the $(P)$-equation  under the ``first order Melnikov'' non-resonance conditions including initialization, iteration and measure estimates.
   Section \ref{sec:3}   is devoted to checking inversion of the linearized operators. Finally, we list the  the proof of some related results for the sake of completeness in section \ref{sec:4}.

\section{Proof of the main results}

The object of this section is to  complete the proof of  the main results.

\subsection{Solution of the $(Q)$-equation}\label{sec:2.1} We will  first solve the $(Q$)-equation via the classical implicit function theorem.
\begin{lemm}\label{lem1}
Let Hypothesis \ref{hy1} hold for some $\hat{\epsilon}\in[0,{\epsilon}_0]$. There exists a neighborhood $(\epsilon_1,\epsilon_2)$ of $\hat{\epsilon}$, and a $C^2$ map
\begin{align*}
v:(\epsilon_1,\epsilon_2)\times\left\{w\in W\cap \mathcal{H}^s:~\|w\|_{s}< r\right\}&\rightarrow H^2_p(0,\pi),\quad(\epsilon,w)\mapsto v(\epsilon,w;\cdot)
\end{align*}
such that $v(\epsilon,w; \cdot)$ solves the ($Q$)-equation with $v(\hat{\epsilon},0;t)=\hat{v(t)}$ and satisfies
\begin{align}\label{C5}
\| v(\epsilon,w; ~\cdot)-v(\epsilon,0; \cdot)\|_{H^2}\leq C\|w\|_s,\quad\|v(\epsilon,0; \cdot)-\hat{v} (\cdot)\|_{H^2}\leq C|\epsilon-\hat{\epsilon}|
\end{align}
for some constant $C>0$.
\end{lemm}
\begin{proof}
It follows from Hypothesis \ref{hy1} that the linearized operator
\begin{equation*}
h\mapsto (ph'')''-\hat{\epsilon}f'_0(\hat{v})h
\end{equation*}
is invertible on $V$. Since $f\in\mathcal{C}_k$, Lemma \ref{lem23} implies that the following map
\begin{equation*}
(\epsilon,w,v)\mapsto (pv'')''-\epsilon \Pi_{V}F(v+w)
\end{equation*}
belongs to  $C^2([\epsilon_1,\epsilon_2]\times (W\cap H^s)\times V;V)$. Therefore, by the implicit function theorem, there is a $C^2$-path $(\epsilon,w)\mapsto v(\epsilon,w;\cdot)$
such that the conclusions of the lemma hold.
\end{proof}

\subsection{Solution of the $(P)$-equation}\label{sec:2.2}
By virtue of a Nash-Moser iterative theorem, the purpose of this subsection is to solve the $(P)$-equation, i.e.,
\begin{align*}
L_{\omega}w=\epsilon \Pi_{W}\mathcal{F}(\epsilon,w),
\end{align*}
where $\mathcal{F}(\epsilon,w):=F(v(\epsilon,w)+w)$. Let $W=W_{N}\oplus W_{N}^{\bot}$, where
\begin{align*}
W_{N}:=\left\{w\in W:w=\sum_{1\leq|l|\leq N}w_{l}(x)e^{\mathrm{i}lt}\right\},\quad W_{N}^{\bot}:=\left\{w\in W:w=\sum_{|l|> N}w_{l}(x)e^{\mathrm{i}lt}\right\}.
\end{align*}
Then corresponding projection operators
$
\mathrm{P}_{N}:W\rightarrow W_{N},~\mathrm{P}^{\bot}_{N}:W\rightarrow W_{N}^{\bot}
$
satisfy
\begin{flalign*}
    &\mathrm{(\bf P1)}\quad\|\mathrm{P}_{N}u\|_{s+\vartheta}\leq N^{\vartheta}\|u\|_{s},\quad \forall u\in \mathcal{H}^s,\forall s,\vartheta\geq0. &\\
    &\mathrm{(\bf P2)}\quad\|\mathrm{P}^{\bot}_{N}u\|_{s}\leq N^{-\vartheta}\|u\|_{s+\vartheta},\quad \forall u\in \mathcal{H}^{s+\vartheta},\forall s,\vartheta\geq0.&
\end{flalign*}
Moreover, if $f \in\mathcal{C}_k$ satisfying $k\geq s'+3$ with $s'\geq s>1/2$, then it follows from Lemmata \ref{lem22}--\ref{lem23} and Lemma \ref{lem1} that composition operator $\mathcal{F}$  has the following standard properties:
\\$\mathrm{(\bf U1)}$(\text{Regularity}.)\quad$\mathcal{F}\in C^2(\mathcal{H}^s;\mathcal{H}^s)$ and $\mathcal{F},\mathrm{D}_{w}\mathcal{F},\mathrm{D}^2_{w}\mathcal{F}$ are bounded on $\{\|w\|_{s}\leq 1\}$, where ${\rm D}_{w} \mathcal{F} $ is the Fr\'{e}chet derivative of $\mathcal{F}$ with respect to $w$;
\\$\mathrm{(\bf U2)}$(\text{Tame}.)\quad$\forall u,h,\mathrm{h}\in H^s$ with $\|u\|_{s}\leq1$,
\begin{align*}
&\|\mathcal{F}(\epsilon,w)\|_{{s'}}\leq C(s')(1+\|w\|_{{s'}}),\quad\|\mathrm{D}_{w}\mathcal{F}(\epsilon,w)h\|_{{s'}}{\leq}C(s')(\|w\|_{{s'}}\|h\|_{s}+\|h\|_{{s'}}),\\
&\|\mathrm{D}^2_{w}\mathcal{F}(\epsilon,w)[h,\mathrm{h}]\|_{s'}{\leq} C(s')(\|w\|_{s'}\|h\|_{s}\|\mathrm{h}\|_{s}+\|h\|_{s'}\|\mathrm{h}\|_{s}+\|h\|_{s}\|\mathrm{h}\|_{s'});
\end{align*}
\\$\mathrm{(\bf U3)}$(\text{Taylor Tame}.)\quad $\forall s\leq s'\leq k-3$, $\forall w,h\in \mathcal{H}^{s'}$ and $\|w\|_{s}\leq1$, $\|h\|_{s}\leq1$,
\begin{align*}
\|\mathcal{F}(\epsilon,w+h)-\mathcal{F}(\epsilon,w)-\mathrm{D}_{w}\mathcal{F}(\epsilon,w)[h]\|_{s}\leq& C\|h\|^2_{s},\\
\|\mathcal{F}(\epsilon,w+h)-\mathcal{F}(\epsilon,w)-\mathrm{D}_{w}\mathcal{F}(\epsilon,w)[h]\|_{{s'}}\leq& C(s')(\|w\|_{{s'}}\|h\|^2_{s}+\|h\|_{s}\|h\|_{{s'}}).
\end{align*}
Let us define the linearized operators $\mathcal{L}_{N}(\epsilon,\omega,w)$ as
\begin{align}\label{D3}
\mathcal{L}_{N}(\epsilon,\omega,w)[h]:=-L_{\omega}h+\epsilon \mathrm{P}_{N}\Pi_{W}\mathrm{D}_{w}\mathcal{F}(\epsilon,w)[h],\quad\forall h\in W_{N},
\end{align}
where $L_{\omega}$ is given by \eqref{C2}. Denote by $\lambda_{j}(\epsilon,w)=\mu^2_j(\epsilon,w),j\in\mathbb{N}^{+}$ the eigenvalues of Euler-Bernoulli beam's problem
\begin{align}\label{D4}
\begin{cases}
(p(x)y'')''-{\epsilon}\Pi_{V}f'(t,x,v(\epsilon,w(t,x))+w(t,x))y=\lambda \rho(x)y,\\
y(0)=y(\pi)=y''(0)=y''(\pi)=0,
\end{cases}
\end{align}
where
\begin{align}\label{D5}
\mu_j(\epsilon,w)=
\begin{cases}
\mathrm{i}\sqrt{-\lambda_{j}(\epsilon,w)},\quad \text{if } \lambda_{j}(\epsilon,w)<0,\\
\sqrt{\lambda_{j}(\epsilon,w)},\quad \text{if }\lambda_{j}(\epsilon,w)>0.
\end{cases}
\end{align}
For fixed $\gamma\in(0,1),\tau\in(1,2)$, define $\Delta^{\gamma,\tau}_{N}(w)$ by
\begin{align}\label{D6}
\Delta^{\gamma,\tau}_{N}(w):=\Bigg\{(\epsilon,\omega)\in(\epsilon_1,\epsilon_2)\times(\gamma,+\infty):& \left|\omega l-\mu_j(\epsilon,w)\right|>\frac{\gamma}{l^{\tau}},\nonumber\\
 &\left|\omega l-{j}\right|>\frac{\gamma}{l^{\tau}},~\forall l=1,2,\cdots,N,j\geq1\Bigg\}.
\end{align}
Note that the non-resonance conditions in \eqref{D6} are trivially satisfied if  $\lambda_{j}(\epsilon,w)<0$.
\begin{lemm}\label{lem2}
Let $(\epsilon,\omega)\in\Delta^{\gamma,\tau}_{N}(w)$ for fixed $\gamma\in(0,1),\tau\in(1,2)$. There exist $K, K(s')>0$ such that if
\begin{equation}\label{D7}
\|w\|_{{s+\sigma}}\leq1\quad\text{with }\sigma:=\frac{\tau(\tau-1)}{2-\tau},
\end{equation}
for $\frac{\epsilon}{\gamma^3\omega}\leq \delta\leq\frac{\mathrm{c}}{L(s')}$ small enough, $\mathcal{L}_{N}(\epsilon,\omega,w)$ is invertible with
\begin{align*}
\left\|\mathcal{L}^{-1}_{N}(\epsilon,\omega,w)h\right\|_{s}\leq&\frac{K}{\gamma\omega}N^{\tau-1}\|h\|_{s},\quad \forall s>{1}/{2},\\
\left\|\mathcal{L}^{-1}_{N}(\epsilon,\omega,w)h\right\|_{{s'}}\leq& \frac{K(s')}{\gamma\omega}N^{\tau-1}\left(\|h\|_{{s'}}+\|w\|_{{s'+\sigma}}\|h\|_{s}\right),\quad\forall s'\geq s>{1}/{2}.
\end{align*}
Note that condition $\delta\leq\frac{\mathrm{c}}{L(s')}$ is to guarantee that Lemma \ref{lem7} holds.
\end{lemm}
Set
\begin{align}\label{D9}
N_n:=\lfloor e^{\mathfrak{c}2^{n}}\rfloor \quad\text{with } \mathfrak{c}=\ln N_0,
\end{align}
where the symbol $\lfloor~\cdot~\rfloor$ denotes the integer part.

Denote by $A_0$ the open set
\begin{align}\label{D8}
A_{0}:=\left\{(\epsilon,\omega)\in(\epsilon_1,\epsilon_2)\times(\gamma,+\infty):~\left|\omega l-\bar{\mu}_{j}\right|>\frac{\gamma}{l^{\tau}},~\forall l=1,\cdots,N_0, j\geq1\right\},
\end{align}
where $\bar{\lambda}_j=\bar{\mu}^2_j,j\geq1$ are the eigenvalues of Euler-Bernoulli beam's problem \eqref{B3}.

Let us state the inductive theorem.
\begin{theo}\label{theo2}
For $\frac{\epsilon}{\gamma^3\omega}\leq \delta_4$ (see Lemma \ref{lem12}) small enough, there exists a sequence of subsets $(\epsilon,\omega)\in A_{n}\subseteq A_{n-1}\subseteq\cdots\subseteq A_1\subseteq A_{0}$, where
\begin{equation*}
A_{n}:=\left\{(\epsilon,\omega)\in A_{n-1}: (\epsilon,\omega)\in\Delta^{\gamma,\tau}_{N_{n}}(w_{n-1})\right\},
\end{equation*}
and a sequence $w_{n}(\epsilon,\omega)\in W_{N_{n}}$
satisfying
\\
$\mathrm{(S1)}_{n\geq0}$\quad $\|w_n\|_{s+\sigma}\leq1$, and $\|\partial_{\omega}w_n\|_{s}\leq\frac{{K}_1\epsilon}{\gamma^2\omega}$, $\|\partial_{\epsilon}w_n\|_{s}\leq\frac{{K}_1}{\gamma\omega}$.
\\
$\mathrm{(S2)}_{n\geq1}$\quad$\|w_k-w_{k-1}\|_{s}\leq\frac{K_2\epsilon}{ \gamma\omega} N_{k}^{-\sigma-1}$, $\|\partial_{\omega}(w_k-w_{k-1})\|_{s}\leq \frac{ K_{3}\epsilon}{\gamma^2\omega}N_{k}^{-1}$, $\|\partial_{\epsilon}(w_k-w_{k-1})\|_{s}\leq \frac{ K_{3}}{\gamma\omega}N_{k}^{-1}$, $\forall 1\leq k\leq n$.
\\
$\mathrm{(S3)}_{n\geq0}$\quad  If $(\epsilon,\omega)\in A_{n}$,  then $w_{n}(\epsilon,\omega)$ is a solution of
\begin{align}\label{D10}
  L_{\omega}w-\epsilon \mathrm{P}_{N_n}\Pi_{W}\mathcal{F}(\epsilon,w)=0, \tag{$P_{N_{n}}$}
\end{align}
where $L_\omega$ is defined in \eqref{C2}.
\\
$\mathrm{(S4)}_{n\geq1}\quad$Setting $B_{n}:=1+\|w_n\|_{s+\kappa}$, $B'_{n}:=1+\|\partial_{\omega}w_n\|_{s+\kappa}$ and $B''_{n}:=1+\|\partial_{\epsilon}w_n\|_{s+\kappa}$, there exists some constant $\bar{K}=\bar{K}(N_0)$ such that
\begin{align*}
&B_{k}\leq C_1\bar{K}N^{{\tau-1+\sigma}}_{k+1},\quad B'_{k}\leq C_2{\bar{K}}{\gamma}^{-1} N^{3\tau+2\sigma-1}_{k+1}, ~B''_{k}\leq C_3{\bar{K}}{(\gamma\omega)}^{-1} N^{3\tau+2\sigma-1}_{k+1},\quad\forall 1\leq k\leq n,
\end{align*}
where $C_i:=C_i(\mathfrak{c},\tau,\sigma),i=1,2,3$, such that for all $(\epsilon,\omega)\in\cap_{n\geq0}A_n$, the sequence $\{w_n=w_n(\epsilon,\omega)\}_{n\geq0}$ converges uniformly in ${s}$-norm to a map
\begin{align*}
w\in C^{1}\left(\cap_{n\geq0}A_n\cap\left\{(\epsilon,\omega):{\epsilon}/{\omega}\leq\delta_4\gamma^3\right\}; W\cap\mathcal{H}^s\right).
\end{align*}
\end{theo}

\subsubsection{Initialization}\label{sec:2.2.1}
Let us check that $(\mathrm{S1})_{0},(\mathrm{S3})_{0}$ hold.
\begin{lemm}\label{lem9}
For all $(\epsilon,\omega)\in A_0$, the operator $\frac{1}{\rho}L_{\omega}$ is invertible with
\begin{align*}
\|(\frac{1}{\rho}L_{\omega})^{-1}h\|_{s}\leq\frac{\tilde{K}N^{\tau-1}_0}{\gamma\omega}\|h\|_{s},\quad\forall s\geq0,\forall h\in W_{N_0}
\end{align*}
for some constan $\tilde{K}>0$.
\end{lemm}
\begin{proof}
The eigenvalues of $\frac{1}{\rho}L_{\omega}$ on $W_{N_0}$ are
\begin{align*}
-\omega^2l^2+\bar{\lambda}_j,\quad \forall1\leq|l|\leq N_0,\forall j\geq1.
\end{align*}
For all $(\epsilon,\omega)\in A_0$, one has
\begin{align*}
|\omega^2l^2-\bar{\lambda}_j|=|\omega l-\bar{\mu}_j||\omega l+\bar{\mu}_j|>\frac{\gamma\omega}{l^{\tau-1}},\quad\forall 1\leq l\leq N_0,~\forall j\geq 1.
\end{align*}
Thus $\frac{1}{\rho}L_{\omega}$ is invertible on $W_{N_0}$ satisfying the conclusion of the lemma.
\end{proof}
\begin{rema}
In the proof of Lemma \ref{lem9}, we apply an equivalent scalar product $(\cdot,\cdot)$ on ${H}^2_{p}(0,\pi)$ as follows
\begin{align*}
(y,z):=\int_{0}^{\pi}py''z''+\rho yz\mathrm{d}t
\end{align*}
with
\begin{align*}
C_1\|y\|_{H^2}\leq\|y\|\leq C_2\|y\|_{H^2},\quad\forall y\in{H}^2_p(0,\pi)
\end{align*}
for some constants $C_1$, $C_2>0$.
\end{rema}
Then solving $(P_{N_0})$ is equivalent to the fixed point problem $w=\mathcal{U}_0(w)$, where
\begin{align*}
\mathcal{U}_0:{W}_{N_0}\rightarrow {W}_{N_0},\quad w\mapsto \epsilon \left(\frac{1}{\rho}L_{\omega}\right)^{-1}\frac{1}{\rho} \mathrm{P}_{N_0}\Pi_{W}\mathcal{F}(\epsilon,w).
\end{align*}
\begin{lemm}\label{lem10}
For all $(\epsilon,\omega)\in A_0$ and $\frac{\epsilon}{\gamma\omega}\leq \delta_1N^{1-\tau}_0\leq\delta$, the map $\mathcal{U}_0$ is a contraction in
\begin{align*}
\mathcal{B}(0,\rho_0):=\left\{w\in W_{N_0}:\|w\|_{s}\leq \rho_0\right\}\quad \text{with } \rho_0:=\frac{\epsilon}{ \gamma\omega}K_1 N^{\tau-1}_0,
\end{align*}
where ${\epsilon \gamma^{-1}K_1 N^{\tau}_0}<r$ ($r$ is given in Lemma \ref{lem1}).
\end{lemm}
\begin{proof}
It follows from Lemma \ref{lem9} and property $(\mathrm{\bf U1})$ that for $\frac{\epsilon}{\gamma\omega}N^{\tau-1}_0\leq\delta_1$ small enough
\begin{align}
\|\mathcal{U}_0(w)\|_{s}{\leq}&\frac{\epsilon\tilde{K}N^{\tau-1}_0}{\gamma\omega}\left\|{1}/{\rho}\right\|_{H^2}\left\|\mathrm{P}_{N_0}\Pi_{W}\mathcal{F}(\epsilon,w)\right\|_s
\leq\frac{\epsilon}{ \gamma\omega}K_1 N^{\tau-1}_0,\nonumber\\
\|\mathrm{D}_w\mathcal{U}_0(w)\|_s=&\|\epsilon (\frac{1}{\rho}L_{\omega})^{-1}\frac{1}{\rho}\mathrm{P}_{N_0}\Pi_{W}\mathrm{D}_w\mathcal{F}(\epsilon,w)\|_s\leq\frac{\epsilon}{ \gamma\omega}K_1 N^{\tau-1}_0\leq\frac{1}{2}.\label{E2}
\end{align}
Thus the map $\mathcal{U}_0$ is a contraction in $\mathcal{B}(0,\rho_0)$.
\end{proof}
Denote by $w_0$ the unique solution of equation $(P_{N_0})$ in $\mathcal{B}(0,\rho_0)$. Then, for $\frac{\epsilon}{\gamma\omega}N^{\tau-1}_0\leq\delta_1$ small enough, applying $(\mathrm{P1})$ and Lemma \ref{lem10} yields
\begin{align}\label{E3}
\|w_0\|_{s+\sigma}\leq1,\quad\|w_0\|_{s+\kappa}=\|\epsilon(\frac{1}{\rho}L_{\omega})^{-1}\frac{1}{\rho}\mathrm{P}_{N_0}\Pi_{W}\mathcal{F}(\epsilon,w_0)\|_{s+\kappa}\leq \bar{K}.
\end{align}
Moreover let us define
\begin{align*}
\mathscr{U}_{0}(\epsilon,\omega,w):=w-\epsilon (\frac{1}{\rho}L_{\omega})^{-1}\frac{1}{\rho}\mathrm{P}_{N_0}\Pi_{W}\mathcal{F}(\epsilon,w).
\end{align*}
It is obvious that $\mathscr{U}_{0}(\epsilon,\omega,w_0)=0$.

By virtue of formula \eqref{E2}, for $\frac{\epsilon}{\gamma\omega}N^{\tau-1}_0\leq\delta_1$ small enough, we obtain
\begin{align*}
\mathrm{D}_{w}\mathscr{U}_{0}(\epsilon,\omega,w_0)=\mathrm{Id}-\epsilon (\frac{1}{\rho}L_{\omega})^{-1}\frac{1}{\rho}\mathrm{P}_{N_0}\Pi_{W}\mathrm{D}_{w}\mathcal{F}(\epsilon,w_0)
\end{align*}
is invertible.  Then the implicit function theorem implies  $w_0\in C^1(A_0;W_{N_0})$, which  gives
\begin{align*}
\partial_{\omega}w_{0}=&\epsilon(\mathrm{Id}-\mathrm{D}_{w}\mathcal{U}_{0}(w_0))^{-1}
\partial_{\omega}(\frac{1}{\rho}L_{\omega})^{-1}\frac{1}{\rho}\mathrm{P}_{N_0}\Pi_{W}\mathcal{F}(\epsilon,w_0),\\
\partial_{\epsilon}w_0=&(\mathrm{Id}-\mathrm{D}_{w}\mathcal{U}_{0}(w_0))^{-1}(\frac{1}{\rho}L_{\omega})^{-1}
\frac{1}{\rho}\mathrm{P}_{N_0}\Pi_{W}\mathcal{F}(\epsilon,w_0)
\end{align*}
by taking the derivatives of $\mathscr{U}_{0}(\epsilon,\omega,w_0)=0$ with respect to $\omega,\epsilon$. Moreover taking the derivative of the identity $(\frac{1}{\rho}L_{\omega})(\frac{1}{\rho}L_{\omega})^{-1}\mathfrak{w}=\mathfrak{w}$ with respect to $\omega$ yields
\begin{align*}
\partial_{\omega}(\frac{1}{\rho}L_{\omega})^{-1}\mathfrak{w}=-(\frac{1}{\rho}L_{\omega})^{-1}
(\frac{2\omega}{\rho}\partial_{tt})(\frac{1}{\rho}L_{\omega})^{-1}\mathfrak{w}.
\end{align*}
Then, in view of \eqref{E2}, $(\mathrm{\bf P1})$ and Lemma \ref{lem9}, we derive
\begin{align*}
\|\partial_{\omega}w_0\|_{s}\leq\frac{K_1\epsilon}{\gamma^2\omega},\quad\|\partial_{\epsilon}w_0\|_{s}\leq\frac{K_1}{\gamma\omega}.
\end{align*}
Above estimates  may give that for $\frac{\epsilon}{\gamma\omega}N^{\tau-1}_0\leq\delta_1$ small enough,
\begin{align}\label{E5}
\|\partial_{\omega}w_0\|_{s+\kappa}\stackrel{(\mathrm{\bf P1})}{\leq}{\bar{K}}{\gamma^{-1}},\quad
\|\partial_{\epsilon}w_0\|_{s+\kappa}\stackrel{(\mathrm{\bf P1})}{\leq}{\bar{K}}{(\gamma\omega)^{-1}}.
\end{align}
Thus we have properties $(\mathrm{S1})_{0}, (\mathrm{S3})_{0}$.

\subsubsection{Iteration}\label{sec:2.2.2}
Assume that we have get a solution $w_{n}\in W_{N_{n}}$ of \eqref{D10}
satisfying properties  $(\mathrm{S1})_{k}$--$(\mathrm{S4})_{k}$ for all $k\leq n$.
Next our purpose is to look for a solution $w_{n+1}\in W_{N_{n+1}}$ of
\begin{align}\label{F1}
L_{\omega}w-\epsilon \mathrm{P}_{N_{n+1}}\Pi_{W}\mathcal{F}(\epsilon,w)=0  \tag{$P_{N_{n+1}}$}
\end{align}
with conditions $(\mathrm{S1})_{n+1}$--$(\mathrm{S4})_{n+1}$.

 For $h\in W_{N_{n+1}}$, denote by
\begin{equation*}
w_{n+1}=w_{n}+h
\end{equation*}
a solution of \eqref{F1}. It follows from the fact 
$L_{\omega}w_{n}=\epsilon \mathrm{P}_{N_{n}}\Pi_{W}\mathcal{F}(\epsilon,w_{n})$ that
\begin{align*}
L_{\omega}(w_{n}+h)-\epsilon \mathrm{P}_{N_{n+1}}\Pi_{W}\mathcal{F}(\epsilon,w_{n}+h)=&L_{\omega}h+ L_{\omega}w_{n}-\epsilon \mathrm{P}_{N_{n+1}}\Pi_{W}\mathcal{F}(\epsilon,w_{n}+h)\\
=&-\mathcal{L}_{N_{n+1}}(\epsilon,\omega,w_n)h+R_{n}(h)+r_{n},
\end{align*}
where
\begin{align*}
&R_{n}(h):=-\epsilon \mathrm{P}_{N_{n+1}}(\Pi_{W}\mathcal{F}(\epsilon,w_n+h)-\Pi_{W}\mathcal{F}(\epsilon,w_n)-\Pi_{W}\mathrm{D}_{w}\mathcal{F}(\epsilon,w_n)[h]),\\
&r_{n}:=\epsilon \mathrm{P}_{N_{n}}\Pi_{W}\mathcal{F}(\epsilon,w_n)-\epsilon \mathrm{P}_{N_{n+1}}\Pi_{W}\mathcal{F}(\epsilon,w_n)=
-\epsilon \mathrm{P}^{\bot}_{N_{n}}\mathrm{P}_{N_{n+1}}\Pi_{W}\mathcal{F}(\epsilon,w_n).
\end{align*}
Since $(\epsilon,\omega)\in A_{n+1}\subseteq A_{n}$ and $\frac{\epsilon}{\gamma\omega}\leq\frac{\epsilon}{\gamma^3\omega}\leq \delta_1N^{1-\tau}_0\leq\delta$, by means of $(\mathrm{S1})_n$ and Lemma \ref{lem2}, we derive that the linearized operator $\mathcal{L}_{N_{n+1}}(\epsilon,\omega,w_n)$ is invertible with
\begin{align}
\|\mathcal{L}^{-1}_{N_{n+1}}(\epsilon,\omega,w_{n})h\|_{s}\leq& \frac{K}{\gamma\omega}N_{n+1}^{\tau-1}\|h\|_{s},\quad\forall s>{1}/{2},\label{F2}\\
\|\mathcal{L}^{-1}_{N_{n+1}}(\epsilon,\omega,w_{n})h\|_{s'}\leq& \frac{K(s')}{\gamma\omega}N_{n+1}^{\tau-1}(\|h\|_{s'}+\|w\|_{s'+\sigma}\|h\|_{s}),\quad \forall s'\geq s>{1}/{2}.\label{F3}
\end{align}
Define a map
\begin{align*}
\mathcal{U}_{n+1}: W_{N_{n+1}}\rightarrow W_{N_{n+1}},\quad h\mapsto\mathcal{L}^{-1}_{N_{n+1}}(\epsilon,\omega,w_n)(R_{n}(h)+r_{n}).
\end{align*}
Then solving $(P_{N_{n+1}})$ is reduced to find  the fixed point of  $h=\mathcal{U}_{n+1}(h)$.
\begin{lemm}\label{lem11}
For $(\epsilon,\omega)\in A_{n+1}$ and $\frac{\epsilon}{\gamma^3\omega}\leq\delta_2\leq\delta_1N^{1-\tau}_0$, there exists $K_2>0$  such that the map $U_{n+1}$ is a contraction in
\begin{align}\label{F4}
\mathcal{B}(0,\rho_{n+1}):=\left\{h\in W_{N_{n+1}}:\|h\|_{s}\leq\rho_{n+1}\right\}\quad \text{with }\rho_{n+1}:=\frac{\epsilon K_2}{ \gamma\omega} N_{n+1}^{-\sigma-1},
\end{align}
where $\frac{\epsilon K_2}{ \gamma\omega} <r$ ($r$ is seen in Lemma \ref{lem1}). Moreover the unique fixed point ${h}_{n+1}(\epsilon,\omega)$ of $\mathcal{U}_{n+1}$ satisfies
\begin{align}\label{F5}
\|{h}_{n+1}\|_{s}\leq\frac{\epsilon}{\gamma\omega} K_2N^{\tau-1}_{n+1}N^{-\kappa}_{n}B_{n}.
\end{align}
\end{lemm}
\begin{proof}
Using properties $(\mathrm{\bf P2})$, $(\mathrm{\bf U2})$--$(\mathrm{\bf U3})$ establishes
\begin{align*}
&\|R_{n}(h)\|_s{\leq}\epsilon C\|h\|^2_s,\quad\|r_n\|_s{\leq}\epsilon C(\kappa)N^{-\kappa}_{n}B_n,
\end{align*}
where $B_n$ is seen in $(\mathrm{S4})_n$. Based on this together with \eqref{F2}, we get
\begin{align}
\|\mathcal{U}_{n+1}(h)\|_{s}
\leq&\frac{\epsilon {K}'}{\gamma\omega}N^{\tau-1}_{n+1}\|h\|^2_s+\frac{\epsilon {K}'}{\gamma\omega}N^{\tau-1}_{n+1}N^{-\kappa}_{n}B_{n}\label{F7}\\
\leq&\frac{\epsilon {K}'}{\gamma\omega}N^{\tau-1}_{n+1}\rho^2_{n+1}+\frac{\epsilon {K}'}{\gamma\omega}N^{\tau-1}_{n+1}N^{-\kappa}_{n}B_{n}.\nonumber
\end{align}
Obviously, one has $\sigma>\tau-1$ according to the fact  $\tau\in(1,2)$. Then using the definition of $\rho_{n+1}$,  \eqref{C6} and $(\mathrm{S4})_n$,
we derive that for $\frac{\epsilon}{\gamma^3\omega}\leq\delta_2$ small enough,
\begin{align}\label{F8}
\frac{\epsilon {K}'}{\gamma\omega}N^{\tau-1}_{n+1}\rho_{n+1}\leq\frac{1}{2},\quad\frac{\epsilon {K}'}{\gamma\omega}N^{\tau-1}_{n+1}N^{-\kappa}_{n}B_{n}\leq\frac{\rho_{n+1}}{2},
\end{align}
which leads to $\|\mathcal{U}_{n+1}(h)\|_{s}\leq \rho_{n+1}$. Moreover taking the derivative  of  $\mathcal{U}_{n+1}$ with respect to $h$ yields
\begin{align}\label{F9}
\mathrm{D}_h\mathcal{U}_{n+1}(h)[\mathfrak{w}]=-\epsilon\mathcal{L}^{-1}_{N_{n+1}}(\epsilon,\omega,w_n)\mathrm{P}_{N_{n+1}}(\Pi_{W}\mathrm{D}_{w}
\mathcal{F}(\epsilon,w_n+h)-\Pi_W\mathrm{D}_{w}\mathcal{F}(\epsilon,w_n))\mathfrak{w}.
\end{align}
For $\frac{\epsilon}{\gamma^3\omega}\leq\delta_2$ small enough, it follows from  $(\mathrm{\bf U1})$--$(\mathrm{\bf U2})$ and $(\mathrm{S1})_{n}$ that
\begin{align}\label{F10}
\|\mathrm{D}_h\mathcal{U}_{n+1}(h)[\mathfrak{w}]\|_{s}\stackrel{ \eqref{F2}}{\leq}\frac{\epsilon{K}'}{\gamma\omega}N^{\tau-1}_{n+1}
\|h\|_{s}\|\mathfrak{w}\|_{s}\leq\frac{\epsilon{K}'}{\gamma\omega}N^{\tau-1}_{n+1}\rho_{n+1}\|\mathfrak{w}\|_{s}\stackrel{\eqref{F8}}\leq\frac{1}{2}\|\mathfrak{w}\|_{s}.
\end{align}
Hence $\mathcal{U}_{n+1}$ is a contraction in $\mathcal{B}(0,\rho_{n+1})$.

Denote by ${h}_{n+1}(\epsilon,\omega)$ the unique fixed point of $\mathcal{U}_{n+1}$. With the help of \eqref{F4}, \eqref{F7}--\eqref{F8}, one has
\begin{align*}
\|h_{n+1}\|_{s}\leq\frac{1}{2}\|h_{n+1}\|_s+\frac{\epsilon {K}'}{\gamma\omega}N^{\tau-1}_{n+1}N^{-\kappa}_{n}B_{n},
\end{align*}
which arrives at \eqref{F5}.
\end{proof}
If $\frac{\epsilon}{\gamma^3\omega}\leq\delta_3$ with $\delta_3\leq\delta_2$ is small enough, setting $h_0=w_0$, applying {Lemmata }\ref{lem10}--\ref{lem11} yields
\begin{equation*}
\|w_{n+1}\|_{s+\sigma}\leq\sum\limits_{i=0}^{n+1}\|h_{i}\|_{s+\sigma}\stackrel{(\mathrm{P1})}{\leq}
\sum\limits_{i=0}^{n+1}N^{\sigma}_{i}\|h_{i}\|_{s}
{\leq}\sum\limits_{i=1}^{n+1}N^{\sigma}_{i}\frac{\epsilon K_2}{\gamma\omega}N^{-\sigma-1}_{i}+N^{\sigma}_0\frac{\epsilon K_1 }{\gamma\omega }N^{\tau-1}_0\leq1.
\end{equation*}
\begin{lemm}\label{lem12}
For $(\epsilon,\omega)\in A_{n+1}$ and $\frac{\epsilon}{\gamma^3\omega}\leq\delta_4\leq\delta_3$, the map $h_{n+1}$ belongs to
$C^1\big(A_{n+1}\cap\{(\epsilon,\omega):{\epsilon}/{\omega}\leq\delta_4\gamma^3\};W_{N_{n+1}}\big)$ and satisfies that for some constant $K_3>0$,
\begin{align*}
\|\partial_{\omega}h_{n+1}\|_{s}\leq\frac{ K_{3}\epsilon}{\gamma^2\omega}N_{n+1}^{-1}, \quad\|\partial_{\epsilon}h_{n+1}\|_{s}\leq\frac{ K_{3}}{\gamma\omega}N_{n+1}^{-1}.
\end{align*}
\end{lemm}
\begin{proof}
Let us define
\begin{align*}
\mathscr{U}_{n+1}(\epsilon,\omega,h):=-L_{\omega}(w_{n}+h)+\epsilon  \mathrm{P}_{N_{n+1}}\Pi_{W}\mathcal{F}(\epsilon,w_n+h).
\end{align*}
Lemma \ref{lem11} shows that $h_{n+1}(\epsilon,\omega)$ is a solution to above equation, i.e.,
\begin{align*}
\mathscr{U}_{n+1}(\epsilon,\omega,h_{n+1})=0,
\end{align*}
which carries out
\begin{align}\label{F13}
\mathrm{D}_{h}\mathscr{U}_{n+1}(\epsilon,\omega,h_{n+1}){=}\mathcal{L}_{N_{n+1}}(\epsilon,\omega,w_{n+1})
\stackrel{\eqref{F9}}{=}\mathcal{L}_{N_{n+1}}(\epsilon,\omega,w_{n})(\mathrm{Id}-\mathrm{D}_{h}\mathcal{U}(h_{n+1})).
\end{align}
By means of \eqref{F10}, the operator $\mathcal{L}_{N_{n+1}}(\epsilon,\omega,w_{n+1})$ is invertible with
\begin{equation}\label{F14}
\|\mathcal{L}^{-1}_{N_{n+1}}(\epsilon,\omega,w_{n+1})\|_{s}\leq\|(\mathrm{Id}-\mathrm{D}_{h}\mathcal{G}(h_{n+1}))^{-1}\mathcal{L}^{-1}_{N_{n+1}}(\epsilon,\omega,w_{n})\|_{s}
\stackrel{\eqref{F2}}{\leq}\frac{2K}{\gamma\omega}N_{n+1}^{\tau-1}.
\end{equation}
Then the implicit function theorem shows  $h_{n+1}\in C^1(A_{n+1};W_{N_{n+1}})$, which  infers
\begin{equation*}
\partial_{\omega,\epsilon}\mathscr U_{n+1}(\epsilon,\omega,h_{n+1})+\mathrm{D}_{h}\mathscr U_{n+1}(\epsilon,\omega,h_{n+1})\partial_{\omega,\epsilon}h_{n+1}=0.
\end{equation*}
Consequently, using $w_{n+1}=w_n+h_{n+1}$  and the fact
$L_{\omega}w_{n}=\epsilon \mathrm{P}_{N_{n}}\Pi_{W}\mathcal{F}(\epsilon,w_{n})$, we obtain
\begin{equation}\label{F15}
\partial_{\omega}h_{n+1}=-\mathcal{L}^{-1}_{N_{n+1}}(\epsilon,\omega,w_{n+1})\partial_{\omega}\mathscr U_{n+1}(\epsilon,\omega,h_{n+1}),
\end{equation}
where
\begin{align}\label{F12}
\partial_{\omega}\mathscr U_{n+1}(\epsilon,\omega,h_{n+1})
=&-2\omega\rho(x)(h_{n+1})_{tt}+\epsilon \mathrm{P}^{\bot}_{N_{n}}\mathrm{P}_{N_{n+1}}\Pi_{W}\mathrm{D}_{w}\mathcal{F}(\epsilon,w_n)\partial_{\omega}w_{n}\nonumber\\
&+\epsilon\mathrm{ P}_{N_{n+1}}(\Pi_{W}\mathrm{D}_{w}\mathcal{F}(\epsilon,w_{n+1})-\Pi_{W}\mathrm{D}_{w}\mathcal{F}(\epsilon,w_n))\partial_{\omega}w_{n},\\
\partial_{\epsilon}\mathscr U_{n+1}(\epsilon,\omega,h_{n+1})
=&\mathrm{P}^{\bot}_{N_n}\mathrm{P}_{N_{n+1}}\mathcal{F}(\epsilon,w_n)+\mathrm{P}_{N_{n+1}}
(\Pi_{W}\mathcal{F}(\epsilon,w_{n+1})-\Pi_{W}\mathcal{F}(\epsilon,w_n))\nonumber\\
&+\epsilon \mathrm{P}^{\bot}_{N_{n}}\mathrm{P}_{N_{n+1}}\Pi_{W}\mathrm{d}_{\epsilon}\mathcal{F}(\epsilon,w_n)+\epsilon \mathrm{P}_{N_{n+1}}
(\Pi_{W}\mathrm{d}_{\epsilon}\mathcal{F}(\epsilon,w_{n+1})-\Pi_{W}\mathrm{d}_{\epsilon}\mathcal{F}(\epsilon,w_n)).\label{F6}
\end{align}
Furthermore Lemma \ref{lem1} and Lemma \ref{lem23} imply
\begin{align}
\|\Pi_{W}\mathrm{d}_{\epsilon}\mathcal{F}(\epsilon,w_n)\|_{s+\kappa}\leq& C(\kappa)\|w_n\|_{s+\kappa}(1+\|\partial_{\epsilon }w_n\|_{s})+C(\kappa)(1+\|\partial_{\epsilon} w_n\|_{s+\kappa}),\label{F19}\\
\|\Pi_{W}\mathrm{d}_{\epsilon}\mathcal{F}(\epsilon,w_{n+1})-\Pi_{W}\mathrm{d}_{\epsilon}\mathcal{F}(\epsilon,w_n)\|_{s}\leq& C(1+\|\partial_{\epsilon} w_n\|_{s})\|h_{n+1}\|_s,\nonumber\\
\|\Pi_{W}\partial_{\epsilon}\mathcal{F}(\epsilon,w_{n+1})-\Pi_{W}\partial_{\epsilon}\mathcal{F}(\epsilon,w_n)\|_{s+\kappa}\leq& C(\kappa)(\|w_n\|_{s+\kappa}\|h_{n+1}\|_{s}+\|h_{n+1}\|_{s+\kappa})(1+\|\partial_{\epsilon} w_n\|_{s})
\nonumber\\&+C(\kappa)(1+\|\partial_{\epsilon} w_n\|_{s+\kappa})\|h\|_{s}.\label{F21}
\end{align}
Then, it follows from  $(\mathrm{\bf P1})$, $(\mathrm{\bf U1})$--$(\mathrm{\bf U2})$ and $(\mathrm{S1})_{n}$ that for $\frac{\epsilon}{\gamma^3\omega}\leq\delta_4$  small enough,
\begin{align}\label{F16}
\|\partial_{\omega}\mathscr{U}_{n+1}(\epsilon,\omega,h_{n+1})\|_{s}\stackrel{\eqref{F5}}{\leq}&
{\epsilon K'}{\gamma^{-1}}N^{\tau+1}_{n+1}N^{-\kappa}_{n}B_{n}+\epsilon K'N^{-\kappa}_{n}B'_{n},\\
\|\partial_{\epsilon}\mathscr{U}_{n+1}(\epsilon,\omega,h_{n+1})\|_{s}\stackrel{\eqref{F5}}{\leq}&
{ K'}N^{\tau-1}_{n+1}N^{-\kappa}_{n}B_{n}+\epsilon K'N^{-\kappa}_{n}B''_{n},\label{F22}
\end{align}
where $B_n,B'_n,B''_n$ are given in $(\mathrm{S4})_n$. Combining above estimates with \eqref{F14}--\eqref{F15}, $(\mathrm{S4})_n$ yields
\begin{align*}
\|\partial_{\omega} h_{n+1}\|_{s}\stackrel{\eqref{C6}}{\leq}\frac{ K_{3}\epsilon}{\gamma^2\omega}N_{n+1}^{-1},\quad\|\partial_{\epsilon} h_{n+1}\|_{s}\stackrel{\eqref{C6}}{\leq}\frac{ K_{3}}{\gamma\omega}N_{n+1}^{-1}.
\end{align*}
This ends the  proof of the lemma.
\end{proof}
Thus we complete the proof of properties  $(\mathrm{S1})_{n+1}$--$(\mathrm{S3})_{n+1}$. Now we are devoted to  giving the upper bounds of ${h}_{n+1},\partial_{\omega}{h}_{n+1}$ in $(s+\kappa)$-norm, i.e., that $(\mathrm{S4})_{n+1}$ holds.

\begin{lemm}\label{lem13}
For $(\epsilon,\omega)\in A_{n+1}$ and $\frac{\epsilon}{\gamma^3\omega}\leq\delta_4$, the first term in $(\mathrm{S4})_{n+1}$ in Theorem \ref{theo2} holds.
\end{lemm}
\begin{proof}
First of all, for $\frac{\epsilon}{\gamma^3\omega}\leq\delta_4$  small enough, we claim
\begin{align}\label{cl2.29}
B_{n+1}\leq (1+N^{\tau-1+\sigma}_{n+1})B_{n}.
\end{align}
Moreover it follows from \eqref{D9} that
\begin{align*}
N^{2}_{n+1}\leq e^{\mathfrak{c}2^{n+2}}<N_{n+2}+1\leq 2N_{n+2}.
\end{align*}
Then \eqref{cl2.29} implies
\begin{align*}
B_{n+1}{\leq}&B_0\prod^{n+1}_{k=1}(1+N^{\tau-1+\sigma}_{k+1})\leq B_0\prod^{n+1}_{k=1}(1+e^{\mathfrak{c}2^k(\tau-1+\sigma)}) \\
\leq&\prod^{+\infty}_{k=1}(1+e^{-\mathfrak{c}2^k(\tau-1+\sigma)})B_0e^{\mathfrak{c}2^{n+2}(\tau-1+\sigma)}\\
\leq& 2^{\tau-1+\sigma}\prod^{+\infty}_{k=1}(1+e^{-\mathfrak{c}2^k(\tau-1+\sigma)})B_0N_{n+2}^{\tau-1+\sigma},
\end{align*}
which shows the first term in $(\mathrm{S4})_{n+1}$ by \eqref{E3}. Now let us prove above claim \eqref{cl2.29}. The definition of $B_{n+1}$ shows
\begin{align}\label{F11}
B_{n+1}\leq1+\|w_{n}\|_{s+\kappa}+\|h_{n+1}\|_{s+\kappa}=B_{n}+\|h_{n+1}\|_{s+\kappa}.
\end{align}
This implies that  we have to  give the upper bound of $\|h_{n+1}\|_{s+\kappa}$. It follows from Lemma \ref{lem11} and $(\mathrm{\bf U2})$--$(\mathrm{\bf U3})$ that
\begin{align*}
&\|r_{n}\|_{s}{\leq}\epsilon C,\quad\|R_{n}(h_{n+1})\|_{s}{\leq}\epsilon C\rho^2_{n+1},\quad\|r_{n}\|_{s+\kappa}\leq \epsilon C(\kappa)B_{n},\\
&\|R_{n}(h_{n+1})\|_{s+\kappa}{\leq}
\epsilon C(\kappa)(\rho^2_{n+1}B_{n}+\rho_{n+1}\|h_{n+1}\|_{s+\kappa}).
\end{align*}
Hence, using the equality $h_{n+1}=\mathcal{L}^{-1}_{N_{n+1}}(\epsilon,\omega,w_{n})(R_{n}(h_{n+1})+r_{n})$, we have
\begin{align*}
\|h_{n+1}\|_{s+\kappa}\stackrel{(\mathrm{\bf P1}),\eqref{F3}}{\leq}\frac{\epsilon K'(\kappa)}{ \gamma\omega}N^{\tau-1+\sigma}_{n+1}B_{n}+\frac{\epsilon K'(\kappa)} {\gamma\omega}N^{\tau-1}_{n+1}\rho_{n+1}\|h_{n+1}\|_{s+\kappa}.
\end{align*}
According to \eqref{F8}, one has that for $\frac{\epsilon}{\gamma\omega}\leq\frac{\epsilon}{\gamma^3\omega}\leq\delta_4$ small enough,
\begin{align}\label{F17}
\|h_{n+1}\|_{s+\kappa}\leq\frac{2\epsilon K'(\kappa)}{ \gamma\omega}N^{\tau-1+\sigma}_{n+1}B_{n}\leq N^{\tau-1+\sigma}_{n+1}B_{n}.
\end{align}
Obviously,  \eqref{cl2.29} follows directly from \eqref{F11}--\eqref{F17}.

\end{proof}
Let us show the upper bound of $\mathcal{L}^{-1}_{N_{n+1}}(\epsilon,\omega,w_{n+1})\mathfrak{w}$ (recall \eqref{F13}) in $(s+\kappa)$-norm.
\begin{lemm}\label{lem15}
For $(\epsilon,\omega)\in A_{n+1}$ and $\frac{\epsilon}{\gamma^3\omega}\leq\delta_4$,  one has that for all $\mathfrak{w}\in W_{N_{n+1}}$,
\begin{align*}
\|\mathcal{L}^{-1}_{N_{n+1}}(\epsilon,\omega,w_{n+1})\mathfrak{w}\|_{s+\kappa}\leq\frac{K_4}{\gamma\omega}N^{\tau-1}_{n+1}\|\mathfrak{w}\|_{s+\kappa}
+\frac{K_4}{\gamma\omega}N^{2\tau-2}_{n+1}(\|w_{n}\|_{s+\kappa+\sigma}+\|h_{n+1}\|_{s+\kappa})\|\mathfrak{w}\|_{s}
\end{align*}
for some constant $K_4>0$.
\end{lemm}
\begin{proof}
Set $\mathscr{L}(h_{n+1}):=(\mathrm{Id}-\mathrm{D}_{h}\mathcal{U}_{n+1}(h_{n+1}))^{-1}\mathfrak{w}$. It is straightforward  that
\begin{align*}
\mathscr{L}(h_{n+1})=\mathfrak{w}+\mathrm{D}_{h}U_{n+1}(h_{n+1})\mathscr{L}(h_{n+1}),\quad
\|\mathscr{L}(h_{n+1})\|_s\stackrel{\eqref{F10}}\leq 2\|\mathfrak{w}\|_{s}.
\end{align*}
With the help of \eqref{F9}, \eqref{F3} and property $(\mathrm{\bf U2})$, we derive
\begin{align*}
\|\mathrm{D}_{h}\mathcal{U}_{n+1}(h_{n+1})\|_{s+\kappa}
\leq\frac{\epsilon K'(\kappa)}{ \gamma\omega}N^{\tau-1}_{n+1}(\|w_{n}\|_{s+\kappa+\sigma}\|h_{n+1}\|_{s}+\|h_{n+1}\|_{s+\kappa}),
\end{align*}
which leads to
\begin{align*}
\|\mathscr{L}(h_{n+1})\|_{s+\kappa}
{\leq}&\|\mathfrak{w}\|_{s+\kappa}+\frac{\epsilon K'(\kappa)}{\gamma\omega} N^{\tau-1}_{n+1}(\|w_{n}\|_{s+\kappa+\sigma}\|h_{n+1}\|_{s}+\|h_{n+1}\|_{s+\kappa})\|\mathfrak{w}\|_{s}\\
&+\frac{\epsilon K'(\kappa)}{ \gamma\omega}N^{\tau-1}_{n+1}\rho_{n+1}\|\mathscr{L}(h_{n+1})\|_{s+\kappa}.
\end{align*}
For $\frac{\epsilon}{\gamma\omega}\leq\frac{\epsilon}{\gamma^3\omega}\leq\delta_4$ small enough, by means of \eqref{F8}, one has
\begin{align*}
\|\mathscr{L}(h_{n+1})\|_{s+\kappa}\leq2\|\mathfrak{w}\|_{s+\kappa}+\frac{2\epsilon K'(\kappa)}{\gamma\omega} N^{\tau-1}_{n+1}(\|w_{n}\|_{s+\kappa+\sigma}\|h_{n+1}\|_{s}+\|h_{n+1}\|_{s+\kappa})\|\mathfrak{w}\|_{s}.
\end{align*}
Hence we get the conclusion  of this  lemma according to \eqref{F3}.
\end{proof}

\begin{lemm}\label{lem14}
For $(\epsilon,\omega)\in A_{n+1}$ and $\frac{\epsilon}{\gamma^3\omega}\leq\delta_4$,
the last two terms in $(\mathrm{S4})_{n+1}$ in Theorem \ref{theo2} hold.
\end{lemm}
\begin{proof}
First of all, for $\frac{\epsilon}{\gamma^3\omega}\leq\delta_4$  small enough, let us check
\begin{align}\label{F23}
&B'_{n+1}\leq(1+N^{\tau-1}_{n+1})B'_{n}+{K'}{\gamma^{-1}}N^{2\tau+\sigma}_{n+1}B_{n},\quad B''_{n+1}\leq(1+N^{\tau-1}_{n+1})B''_{n}+{K'}{(\gamma\omega)^{-1}}N^{2\tau+\sigma}_{n+1}B_{n}.
\end{align}
We only show the upper bound of $B'_n$, while   the upper bound of $B''_n$ can be proved in a  similar manner as employed on the one of $B'_n$.  Denote  $\alpha_1=\tau-1$, $\alpha_2=2\tau+\sigma$, $\alpha_3=\tau-1+\sigma$. The first formula above leads to
\begin{align*}
B'_{n}\leq \mathcal{S}_1+\mathcal{S}_2,\quad\text{with}\quad\mathcal{S}_1=B'_{0}\prod^{n+1}_{k=1}\left(1+N^{\alpha_1}_{k}\right),
\quad\mathcal{S}_2=\sum^{n+1}_{k=1}\mathcal{S}_{2,k},
\end{align*}
where $\mathcal{S}_{2,1}=\frac{K'}{\gamma}N^{\alpha_2}_{n+1}B_{n} $  and
\begin{align*}
\mathcal{S}_{2,k}=\frac{K'}{\gamma}\left(\prod_{j=2}^k\left(1+N^{\alpha_1}_{n+1-(j-2)}\right)\right)N^{\alpha_2}_{n+1-(k-1)}B_{n+1-k}, \quad \forall2\leq k\leq n+1.
\end{align*}
By proceeding as the proof of Lemma \ref{lem13} on  the upper bound on $B_n$, we obtain
 \[\mathcal{S}_1\leq C(\mathfrak{c},\tau,\sigma)B'_{0}N^{\alpha_1}_{n+2}.\]
It follows from the first term in $\mathrm{(S4)}_{n}$ that
\begin{equation*}
\mathcal{S}_{2,1}\leq{K''\gamma^{-1}B_0}e^{\alpha_2\mathfrak{c}2^{n+1}}e^{\alpha_3\mathfrak{c}2^{n+1}}
=K''\gamma^{-1}B_0e^{(\alpha_2+\alpha_3)\mathfrak{c}2^{n+1}}\leq C_1'\gamma^{-1}B_0N^{\alpha_2+\alpha_3}_{n+2}.
\end{equation*}
On the other hand, one has
\begin{align*}
\sum\limits_{k=2}^{n+1} \mathcal{S}_{2,k}&\leq K''\gamma^{-1}B_0\sum\limits_{k=2}^{n+1}e^{\alpha_1\mathfrak{c}{(2^{n+2}-2^{n+3-k})}}
e^{\alpha_2\mathfrak{c}2^{n+2-k}}e^{{\alpha_3}\mathfrak{c}{2^{n+2-k}}}\\
&\leq K''\gamma^{-1}B_0e^{\alpha_1\mathfrak{c}{2^{n+2}}}
\sum\limits_{k=2}^{n+1}e^{{(-\alpha_1+\alpha_2+\alpha_3)}\mathfrak{c}{2^{n+3-k}}}\\
&\leq K''\gamma^{-1}B_0e^{(\alpha_2+\alpha_3)\mathfrak{c}{2^{n+2}}}\\
&\leq C_1'\gamma^{-1}B_0N^{\alpha_2+\alpha_3}_{n+2}.
\end{align*}
Thus formulae \eqref{E3}--\eqref{E5} reads the upper bound of $B'_n$.

We now are devoted to verifying \eqref{F23}. It follows from the definition of $B'_{n},B''_{n}$ that
\begin{align}\label{F18}
B'_{n+1},B''_{n+1}\leq1+\|\partial_{\omega,\epsilon}w_{n}\|_{s+\kappa}+\|\partial_{\omega,\epsilon}h_{n+1}\|_{s+\kappa}.
\end{align}
Thus  we give the upper bound of $\partial_{\omega,\epsilon}h_{n+1}$ in $(s+\kappa)$-norm.

By formula \eqref{F15} and Lemma \ref{lem15}, we may obtain
\begin{align*}
\|\partial_{\omega,\epsilon}h_{n+1}\|_{s+\kappa}\leq&\frac{K_4}{\gamma\omega}N^{\tau-1}_{n+1}\|\partial_{\omega,\epsilon}\mathscr{U}_{n+1}(\epsilon,\omega,h_{n+1})\|_{s+\kappa}\nonumber\\
&+\frac{K_4}{\gamma\omega}N^{2\tau-2}_{n+1}(\|w_{n}\|_{s+\kappa+\sigma}
+\|h_{n+1}\|_{s+\kappa})\|\partial_{\omega,\epsilon}\mathscr{U}_{n+1}(\epsilon,\omega,h_{n+1})\|_{s}.
\end{align*}
Let us show the upper bound of $\partial_{\omega,\epsilon}\mathscr{U}_{n+1}(\epsilon,\omega,h_{n+1})$ in $(s+\kappa)$-norm. Then, for $\frac{\epsilon}{\gamma^3\omega}\leq\delta_4$  small enough, applying $(\mathrm{\bf U1})\text{--}(\mathrm{\bf U2}), (\mathrm{S1})_{n}$ and \eqref{F17} yields
\begin{align*}
\|\partial_{\omega}\mathscr{U}_{n+1}(\epsilon,\omega,h_{n+1})\|_{s+\kappa}\stackrel{\eqref{F12}}{\leq}& {C'(\kappa)}{\omega}N^{\tau+1+\sigma}_{n+1}B_{n}+\epsilon C'(\kappa)B'_{n},\\
\|\partial_{\epsilon}\mathscr{U}_{n+1}(\epsilon,\omega,h_{n+1})\|_{s+\kappa}\stackrel{\eqref{F6}\text{--}\eqref{F21}}{\leq} & {C'(\kappa)}N^{\tau-1+\sigma}_{n+1}B_{n}+\epsilon C'(\kappa)B''_{n}.
\end{align*}
Moreover, due to \eqref{F16}--\eqref{F22} and $(\mathrm{S4})_{n}$, we have
\begin{align*}
\|\partial_{\omega}\mathscr{U}_{n}(\epsilon,\omega,h_n)\|_{s}\stackrel{\eqref{C6}}{\leq}{\epsilon C'}{\gamma^{-1}},\quad\|\partial_{\epsilon}\mathscr{U}_{n+1}(\epsilon,\omega,h_{n+1})\|_{s}\stackrel{\eqref{C6}}{\leq}{C'}.
\end{align*}
Hence one has that for $\frac{\epsilon}{\gamma^3\omega}\leq\delta_4$  small enough,
\begin{align*}
\|\partial_{\omega}h_{n+1}\|_{s+\kappa}\leq \frac{K'}{\gamma}N^{2\tau+\sigma}_{n+1}B_{n}+\frac{\epsilon K'}{\gamma\omega}N^{\tau-1}_{n+1}B'_{n},\quad
\|\partial_{\epsilon}h_{n+1}\|_{s+\kappa}\leq \frac{K'}{\gamma\omega}N^{2\tau+\sigma}_{n+1}B_{n}+\frac{\epsilon K'}{\gamma\omega}N^{\tau-1}_{n+1}B''_{n},
\end{align*}
which gives rise to \eqref{F23} because of \eqref{F18}.
\end{proof}

\subsubsection{Whitney extension}\label{2.2.3}
Let us define
\begin{align}
&\hat{A}_n:=\left\{(\epsilon,\omega)\in A_n,\quad\mathrm{dist}((\epsilon,\omega),\partial A_n)>\frac{\gamma_0\gamma^4}{N^{\tau+1}_n}\right\}\label{G1}
\\
&\tilde{A}_n:=\left\{(\epsilon,\omega)\in A_n,\quad\mathrm{dist}((\epsilon,\omega),\partial A_n )>\frac{2\gamma_{0}\gamma^4}{N^{\tau+1}_n}\right\}\subset \hat{A}_n,\label{G3}
\end{align}
where $A_n$ is given in Theorem \ref{theo2}.

Define a $C^{\infty}$ cut-off function $\varphi_n:~A_0\rightarrow[0,1]$ as
\begin{align*}
\varphi_n:=
\begin{cases}
1\quad\quad\text{if}\quad(\omega,\epsilon)\in\tilde{A}_n,\\
0\quad\quad\text{if}\quad(\omega,\epsilon)\in\hat{A}_n,
\end{cases}
\quad \text{with }|\partial_{\omega,\epsilon}\varphi_n|\leq C{N^{\tau+1}_{n}}/{(\gamma_0\gamma^4)},
\end{align*}
where $A_0$ is defined by \eqref{D8}, $\gamma_{0}$ will be given in Lemma \ref{lem16}.
Then,
\[\tilde{h}_n:=\varphi_nh_n\in C^1(A_0\cap\{(\epsilon,\omega):{\epsilon}/{\omega}\leq\delta_4\gamma^3\};W_{N_n}).\]
 For $n\in\mathbb{N}^{+}$,
it follows from  \eqref{F4}--\eqref{F5}, \eqref{C6}, $(\mathrm{S4})_{n}$ and Lemma \ref{lem12}
 that
\begin{align*}
&\|\tilde{h}_{0}\|_{s}\stackrel{\text{Lemma }\ref{lem10}}{\leq}\frac{\tilde{C}\epsilon}{\gamma\omega},\quad
\|\partial_{\omega}\tilde{h}_{0}\|_{s}\stackrel{(\mathrm{S1})_0}\leq\frac{\tilde{C}(\gamma_0)\epsilon}{\gamma^{5}\omega},
\quad\|\partial_{\epsilon}\tilde{h}_{n}\|_{s}\stackrel{(\mathrm{S1})_0}\leq\frac{\tilde{C}(\gamma_0)}{\gamma^{5}\omega},\\
&\|\tilde{h}_{n}\|_{s}\leq\frac{\tilde{C}\epsilon}{\gamma\omega}N^{-\sigma-1}_{n},\quad
\|\partial_{\omega}\tilde{h}_{n}\|_{s}\leq\frac{\tilde{C}(\gamma_0)\epsilon}{\gamma^{5}\omega}N^{-1}_{n},
\quad\|\partial_{\epsilon}\tilde{h}_{n}\|_{s}\leq\frac{\tilde{C}(\gamma_0)}{\gamma^{5}\omega}N^{-1}_{n}.
\end{align*}
Moreover $\tilde{w}_n=\sum_{k=0}^{n}\tilde{h}_k$ is an extension of ${w}_n$ satisfying $\tilde{w}_n(\epsilon,\omega)={w}_n(\epsilon,\omega)$
 for all $(\epsilon,\omega)\in\tilde{A}_n\cap\{(\epsilon,\omega):{\epsilon}/{\omega}\leq\delta_4\gamma^3\}$.
Then $\tilde{w}(\epsilon,\omega)$ belongs to $C^1(A_0\cap\{(\epsilon,\omega):{\epsilon}/{\omega}\leq\delta_5\gamma^3\};W\cap\mathcal{H}^s)$ with $\delta_5\leq\delta_4$ and satisfies
\begin{align}\label{G6}
\|\tilde{w}\|_{s}\leq\frac{K\epsilon}{\gamma\omega}<r,
\quad\|\partial_{\omega}\tilde{w}\|_s\leq\frac{K(\gamma_0)\epsilon}{\gamma^5\omega},
\quad\|\partial_{\epsilon}\tilde{w}\|_s\leq\frac{K(\gamma_0)}{\gamma^5\omega}.
\end{align}
Since $N_{n}\leq e^{\mathfrak{c}2^{n}}<N_{n}+1<2N_{n}$,
formulae \eqref{F5} and \eqref{C6} give that for $n\geq1$,
\begin{align}
\|\tilde{w}-\tilde{w}_{n}\|_{s}&\leq\sum\limits_{k\geq n+1}\frac{\tilde{C}\epsilon}{\gamma\omega}N^{-\tau-\sigma-4}_{k}{\leq}\sum\limits_{k\geq n+1}\frac{\tilde{C}'\epsilon}{\gamma\omega}e^{-(\tau+\sigma+2)\mathfrak{c}2^{k}} \nonumber\\
&\leq\frac{\tilde{C}''\epsilon}{\gamma\omega}e^{-(\tau+\sigma+2)\mathfrak{c}2^{n}}
\leq\frac{\bar{C}\epsilon}{\gamma\omega}N^{-(\tau+\sigma+2)/2}_{n+1}.\label{G7}
\end{align}
Denote by $\lambda_{j}(\epsilon,\tilde{w})=\mu^2_j(\epsilon,\tilde{w}),j\in\mathbb{N}^{+}$ the eigenvalues of Euler-Bernoulli beam's problem
\begin{align*}
\begin{cases}
(py'')''-{\epsilon}\Pi_{V}f'(v(\epsilon,\tilde{w})+\tilde{w})y=\lambda \rho y,\\
y(0)=y(\pi)=y''(0)=y''(\pi)=0.
\end{cases}
\end{align*}
Moreover let us define
\begin{align}\label{G9}
B_{\gamma}:=\bigg\{&(\epsilon,\omega)\in(\epsilon_1,\epsilon_2)\times(2\gamma,+\infty):~\left|\omega l-\bar{\mu}_j\right|>\frac{2\gamma}{l^{\tau}},
~\forall l=1,\cdots,N_0,~\forall j\geq1,\nonumber\\
&\frac{\epsilon}{\omega}\leq\delta_7\gamma^5,~\left|\omega l-{j}\right|>\frac{2\gamma}{l^{\tau}},
\left|\omega l-\mu_j(\epsilon,\tilde{w})\right|>\frac{2\gamma}{l^{\tau}},~\forall l\geq1,~\forall j\geq1
\bigg\}.
\end{align}

\begin{lemm}\label{lem16}
If $\frac{\epsilon}{\gamma^2\omega}\leq\frac{\epsilon}{\gamma^3\omega}\leq\delta_6\leq\delta_5$ is  small enough, we have that for some $\gamma_0>0$,
\begin{align*}
B_{\gamma}\subseteq\tilde{A}_n\subset A_n,\quad n\geq0.
\end{align*}
\end{lemm}
\begin{lemm}\label{lem17}
For all $(\epsilon,w),(\bar{\epsilon},\bar{w})\in(\epsilon_1,\epsilon_2)\times\left\{W\cap H^s:\|w\|_{s}<r \right\}$, the eigenvalues of \eqref{D4} satisfy that for some constant $\nu>0$,
\begin{align}\label{G10}
|\lambda_{j}(\epsilon,w)-\lambda_{j}(\bar{\epsilon},\bar{w})|\leq \nu(|\epsilon-\bar{\epsilon}|+\|w-\bar{w}\|_s),\quad \forall j\geq1.
\end{align}
\end{lemm}
\begin{proof}
Define
\begin{align*}
g(t,x)=-{\epsilon}\Pi_{V}f'(t,x,v(\epsilon,w(t,x))+w(t,x))\in {H}^2_p(0,\pi)\hookrightarrow C^{1}[0,\pi].
\end{align*}
Let $\psi_{j}(g),j\in\mathbb{N}^{+}$ denote the eigenfunctions with respect to $\lambda_{j}(g)$  of problem \eqref{D4}. Since the coefficients in problem \eqref{D4} satisfy the assumptions of \cite[Theorem 4.4]{Zettl1997dependence}, then it yields
\begin{align*}
\mathrm{D}_{g}\lambda_{j}(g)[h]=-\int^{\pi}_{0}(\psi_{j}(g))^2 h \mathrm{d}x.
\end{align*}
Then, one has
\begin{align*}
|\lambda_{j}(g)-\lambda_{j}(\bar{g})|=&\left|\int^{1}_{0}\int^{\pi}_{0}(\psi_{j}(g+\mathfrak{v}(\bar{g}-g)))^2(g-\bar{g}) \mathrm{d}x\mathrm{d}\mathfrak{v}\right|\\
\leq&\max_{\mathfrak{v}\in[0,1]}\left|\int^{\pi}_{0}(\psi_{j}(g+\mathfrak{v}(\bar{g}-g)))^2(g-\bar{g}) \mathrm{d}x\right|\\
\leq&\|(g-\bar{g})/\rho\|_{L^{\infty}(\mathbb{T};{H}^2_p(0,\pi))}\max_{\mathfrak{v}\in[0,1]}
\left|\int^{\pi}_{0}(\psi_{j}(g+\mathfrak{v}(\bar{g}-g)))^2\rho\mathrm{d}x\right|\\
\leq&\|(g-\bar{g})/\rho\|_{H^2(0,\pi)}\stackrel{\text{Lemmata }\ref{lem23},~\ref{lem1}}{\leq} \nu(|\epsilon-\bar{\epsilon}|+\|w-\bar{w}\|_s),
\end{align*}
which completes the proof of the lemma.
\end{proof}
Moreover the non-degeneracy of $\hat{v}=v(\hat{\epsilon},0)$ means that $\lambda_{j}(\hat{\epsilon},0)\neq0$ for all $j\geq1$. Then, formula \eqref{G10} implies
\begin{align*}
\nu_0:=\inf\left\{|\lambda_j(\epsilon,\omega)|:~j\geq1,\epsilon\in(\epsilon_1,\epsilon_2),\|w\|_s< r\right\}>0.
\end{align*}
In fact, here we may take that $|\epsilon_2-\epsilon_1|,r$ are smaller than ones  in Lemma \ref{lem1} and use such a  number $\nu_0$ for the proof of Lemma \ref{lem16}.
\begin{proof}[\bf {Proof of Lemma \ref{lem16}}]
Obviously, we have $\tilde{A}_n\subset A_n,\forall n\in\mathbb{N}$.
 Moreover we claim that\\
 {\bf (F1):}~ If $\frac{\epsilon}{\gamma^2\omega}\leq\delta_6$ small enough, then there exists $\gamma_0>0$ such that for all $(\epsilon,\omega)\in B_{\gamma}$,
\begin{align*}
{\mathcal{B}}\left((\epsilon,\omega),\frac{2\gamma_0\gamma^4}{N^{\tau+1}_n}\right)\subseteq A_n.
\end{align*}
This implies that $(\epsilon,\omega)$ may belong to $\tilde{A}_n$ for all $n\in\mathbb{N}$.

Let us check claim {\bf (F1)} by induction.
If $\gamma_0\leq\frac{1}{2}$, for all $(\bar\epsilon,\bar\omega)\in{\mathcal{B}}\left((\epsilon,\omega),\frac{2\gamma_0\gamma^4}{N^{\tau+1}_0}\right)$,  one has
\begin{align*}
\left|\bar\omega l-\bar{\mu}_j\right|\geq\left|\omega l-\bar{\mu}_j\right|-\left|\omega-\bar\omega\right|l>\frac{2\gamma}{l^{\tau}}-\frac{2\gamma_0\gamma^4}{N^{\tau+1}_0}l
\geq\frac{\gamma}{l^\tau}+\frac{\gamma}{N^\tau_0}-\frac{2\gamma_0\gamma^4}{N^{\tau}_0}\geq\frac{\gamma}{l^{\tau}},\quad \forall1\leq l\leq N_0,
\end{align*}
which  gives rise to $(\bar\epsilon,\bar\omega)\in A_n$.

Suppose  that
\[{\mathcal{B}}\left((\epsilon,\omega),\frac{2\gamma_0\gamma^4}{N^{\tau+1}_n}\right)\subseteq A_n.\]
It is clear that $(\epsilon,\omega)\in\tilde{A}_n$, which leads to $\tilde{w}_n(\epsilon,\omega)={w}_n(\epsilon,\omega)$.

Finally, we show that claim {\bf (F1)} holds at $(n+1)$-th step. For $\gamma_0\leq\frac{1}{2}$, a similar argument yields that for all $(\bar\epsilon,\bar\omega)\in\mathcal{B}\left((\epsilon,\omega),\frac{2\gamma_0\gamma^4}{N^{\tau+1}_{n+1}}\right)$,
\begin{align*}
\left|\omega_1l-{j}\right|\geq\left|\omega l-{j}\right|-\left|\omega-\bar\omega\right|l>\frac{2\gamma}{l^{\tau}}-\frac{2\gamma_0\gamma^4}{N^{\tau+1}_{N_{n+1}}}l\geq \frac{\gamma}{l^\tau}
+\frac{\gamma}{N^\tau_{n+1}}-\frac{2\gamma_0\gamma^4}{N^{\tau}_{n+1}}\geq\frac{\gamma}{l^{\tau}},\quad \forall 1\leq l\leq N_{n+1}.
\end{align*}
For brevity, denote $\mu^2_{j,n}(\bar\epsilon,\bar\omega)=\lambda_{j,n}(\bar\epsilon,\bar\omega):=\lambda_{j}(\bar\epsilon,w_n(\bar\epsilon,\bar\omega))$, $\tilde{\mu}^2_j(\epsilon,\omega)=\tilde{\lambda}_j(\epsilon,\omega):=\lambda_{j}(\epsilon,\tilde{w}(\epsilon,\omega))$.
Then, it follows from \eqref{G10}, $(\mathrm{S1})_n$ and \eqref{G7} that
\begin{align*}
\left|\mu_{j,n}(\bar\epsilon,\bar\omega)-{\tilde{\mu}_{j}(\epsilon,\omega)}\right|=&\frac{|\lambda_{j,n}(\bar\epsilon,\bar\omega)
-\tilde{\lambda}_{j}(\epsilon,\omega)|}{\left|{\mu_{j,n}(\bar\epsilon,\bar\omega)}\right|+\left|{\tilde{\mu }_{j}(\epsilon,\omega)}\right|}\leq
\frac{1}{\sqrt{\nu_0}}{|\lambda_{j,n}(\bar\epsilon,\bar\omega)
-\tilde{\lambda}_{j}(\epsilon,\omega)|}\\
\leq&\frac{\nu}{\sqrt{\nu_0}}(|\bar\epsilon-\epsilon|+\|w_n(\bar\epsilon,\bar\omega)-\tilde{w}(\epsilon,\omega)\|_s)\\
\leq&\frac{\nu}{\sqrt{\nu_0}}|\bar\epsilon-\epsilon|+\frac{\nu}{\sqrt{\nu_0}}\|w_n(\bar\epsilon,\bar\omega)-{w}_n(\bar\epsilon,\omega)\|_s\\
&+\frac{\nu}{\sqrt{\nu_0}}\|w_n(\bar\epsilon,\omega)-{w}_n(\epsilon,\omega)\|_s
+\frac{\nu}{\sqrt{\nu_0}}\|\tilde{w}_n(\epsilon,\omega)-\tilde{w}(\epsilon,\omega)\|_s\\
\leq&\frac{\nu}{\sqrt{\nu_0}}\left(\frac{2\gamma_0\gamma^4}{N^{\tau+1}_{n+1}}+\frac{2K_1}{\gamma^2\omega}\frac{2\gamma_0\gamma^4}{N^{\tau+1}_{n+1}}+
\frac{\bar{C}\epsilon}{\gamma\omega}N^{-(\tau+\sigma+2)/2}_{n+1}\right).
\end{align*}
Since  $-(\tau+\sigma+2)/2\leq{-\tau}$,  we infer  that for $\gamma_0,\frac{\epsilon}{\gamma^2\omega}$ small enough,
\begin{align*}
\left|\mu_{j,n}(\bar\epsilon,\bar\omega)-{\tilde{\mu}_{j}(\epsilon,\omega)}\right|\leq\frac{\gamma}{2l^{\tau}}.
\end{align*}
Consequently, for all $(\epsilon_1,\omega_1)\in\mathcal{B}\left((\epsilon,\omega),\frac{2\gamma_0\gamma^4}{N^{\tau+1}_{n+1}}\right),$
we can obtain that for $\gamma_0,\frac{\epsilon}{\gamma^2\omega}$ small enough,
\begin{align*}
\left|\omega_1l-\mu_{j,n}(\bar\epsilon,\bar\omega)\right|&\geq\left|\omega l-\tilde{\mu}_{j}(\epsilon,\omega)\right|-|\omega-\bar\omega|l-\left|\mu_{j,n}(\bar\epsilon,\bar\omega)-{\tilde{\mu}_{j}(\epsilon,\omega)}\right|\\
&>\frac{2\gamma}{l^{\tau}}-\frac{2\gamma_0\gamma^4}{N^{\tau+1}_{n+1}}l-\frac{\gamma}{2l^{\tau}}>\frac{\gamma}{l^{\tau}},\quad\forall l=1,\cdots,N_{n+1}.
\end{align*}
The proof  is completed.
\end{proof}
Let $\Omega:=(\epsilon',\epsilon'')\times(\omega',\omega'')$ stand for a rectangle contained in $(\epsilon_1,\epsilon_2)\times(2\gamma,+\infty)$ and set
\begin{align}\label{G13}
\nu_1:=&\inf\left\{\left|{\mu_{j+1}(\epsilon,\omega)}-\mu_j(\epsilon,\omega)\right|:j\geq1,\epsilon\in(\epsilon_1,\epsilon_2),\|w\|_s< r\right\}>0,\\
\nu_2:=&\inf\left\{\left|{\mu_{j+1}(\epsilon,\omega)}-{{\mu}_j(\epsilon,\omega)}\right|:j\geq1,(\epsilon,\omega)\in B_\gamma\right\},\nonumber
\end{align}
where $\mu^2_j(\epsilon,\omega)=\lambda_{j}(\epsilon,\omega):=\lambda_{j}(\epsilon,w(\epsilon,\omega))$. The proof of formula \eqref{G13} will be given in the appendix.
 Moreover, we assume $\omega''-\omega'\geq1$.
\begin{lemm}\label{lem18}
For fixed $\epsilon\in (\epsilon',\epsilon'')$, the measure estimate on $B_{\gamma}(\epsilon)$ satisfies
\begin{align}\label{G11}
\mathrm{meas}(B_{\gamma}(\epsilon)\cap(\omega',\omega''))\geq(1-\mathcal{Q}\gamma)(\omega''-\omega')
\end{align}
for some constant $\mathcal{Q}>0$, where $B_{\gamma}(\epsilon):=\left\{\omega:(\epsilon,\omega)\in B_\gamma\right\}$. Furthermore
\begin{align*}
\mathrm{meas}(B_{\gamma}\cap\Omega)\geq(1-\mathcal{Q}\gamma)\mathrm{meas}(\Omega)=(1-\mathcal{Q}\gamma)(\omega''-\omega')(\epsilon''-\epsilon').
\end{align*}
\end{lemm}
\begin{proof}
Let $(B_{\gamma}(\epsilon))^{c}$ denote the complementary set of $B_{\gamma}(\epsilon)$. Using the definition of $B_{\gamma}$ yields
\begin{align*}
(B_{\gamma}(\epsilon))^{c}\subseteq\mathfrak{R}^1(\epsilon)\cup\mathfrak{R}^2\cup\mathfrak{R}^3,
\end{align*}
where $\mathfrak{R}^1(\epsilon)=\bigcup\limits_{l\geq1,j\geq1}\mathfrak{R}^1_{l,j}(\epsilon)$, $\mathfrak{R}^2=\bigcup\limits_{l\geq1,j\geq1}\mathfrak{R}^2_{l,j}$, $\mathfrak{R}^3=\bigcup\limits_{l\geq1,j\geq1}\mathfrak{R}^3_{l,j}$, and
\begin{align*}
\mathfrak{R}^1_{l,j}(\epsilon):=&\left\{\omega\in(\omega',\omega''):\left|\omega l-\tilde{\mu}_{j}(\epsilon,\omega)\right|\leq\frac{2\gamma}{l^{\tau}}\right\},\\
\mathfrak{R}^2_{l,j}:=&\left\{\omega\in(\omega',\omega''):\left|\omega l-\bar{\mu}_{j}\right|\leq\frac{2\gamma}{l^{\tau}}\right\},\\
\mathfrak{R}^3_{l,j}:=&\left\{\omega\in(\omega',\omega''):\left|\omega l-{j}\right|\leq\frac{2\gamma}{l^{\tau}}\right\}.
\end{align*}
We now show the upper bound of $\mathrm{meas}(\mathfrak{R}^1(\epsilon))$. It follows from \eqref{G10}, \eqref{G6} and the definition of $\nu_0$ that
\begin{align*}
\left|\tilde{\mu}_{j}(\epsilon,\omega_1)-\tilde{\mu}_{j}(\epsilon,\omega)\right|&=\frac{\left|\tilde{\lambda}_{j}(\epsilon,\omega_1)
-\tilde{\lambda}_{j}(\epsilon,\omega)\right|}{\left|{\tilde{\mu}_{j}(\epsilon,\omega_1)}\right|+\left|{\tilde{\mu}_{j}(\epsilon,\omega)}\right|}\leq
\frac{1}{\sqrt{\nu_0}}{\left|\tilde{\lambda}_{j}(\epsilon,\omega_1)
-\tilde{\lambda}_{j}(\epsilon,\omega)\right|}\\
&\leq\frac{\nu}{\sqrt{\nu_0}}\|\tilde{w}(\epsilon,\omega_1)-\tilde{w}(\epsilon,\omega)\|_s\leq\frac{ \epsilon\nu K(\gamma_0)}{\sqrt{\nu_0}\gamma^5\omega}|\omega_1-\omega|,
\end{align*}
which leads to
\begin{align*}
\left|\partial_{\omega}\tilde{\mu}_j(\epsilon,\omega)\right|\leq\frac{ \epsilon\nu K(\gamma_0)}{\sqrt{\nu_0}\gamma^5\omega}.
\end{align*}
Let $g(\omega):=\omega l-\tilde{\mu}_j(\epsilon,\omega)$. Hence it is clear that for $\frac{\epsilon}{\gamma^5\omega}\leq\delta_7$  small enough,
\begin{align*}
\partial_{\omega}g(\omega)=l-\partial_{\omega}\tilde{\mu}_j(\epsilon,\omega)\geq{l}/{2},\quad\forall l\geq1,
\end{align*}
which carries out
\begin{align*}
\mathrm{meas}(\mathfrak{R}^1_{l,j}(\epsilon))\leq\frac{|g(\omega_1)-g(\omega_2)|}{\min|\partial_{\omega}g(\omega)|}\leq\frac{8\gamma}{l^{\tau+1}}.
\end{align*}
For fixed $l$, we get
\begin{align*}
\omega'l-\frac{2\gamma}{l^\tau}\leq\tilde{\mu}_j(\epsilon,\omega)\leq\omega''l+\frac{2\gamma}{l^\tau}\quad\text{if}\quad\mathfrak{R}^1_{l,j}(\epsilon)\neq\emptyset.
\end{align*}
Since $B_\gamma\subseteq\tilde{A}_n$ (recall Lemma \ref{lem17}), one has $w(\epsilon,\omega)=\tilde{w}(\epsilon,\omega)$, which leads to ${{\lambda_{j}(\epsilon,\omega)}}={{\tilde{\lambda}_{j}(\epsilon,\omega)}}$ on $B_\gamma\cap\cap\{(\epsilon,\omega):{\epsilon}/{\omega}\leq\delta_7\gamma^5\}$.
Then, formula  \eqref{G6} implies $\nu_2\geq\nu_1>0$  for $\frac{\epsilon}{\gamma^5\omega}\leq\delta_7$\ small enough. Thus one has
\begin{align*}
\sharp j\leq\frac{1}{\nu_1}(l(\omega''-\omega')+\frac{4\gamma}{l^{\tau}})+1,
\end{align*}
where $\sharp j$ denotes the number of $j$. Consequently, we obtain
\begin{align*}
\mathrm{meas}(\mathfrak{R}^1(\epsilon))\leq\sum\limits_{l=1}^{+\infty}\frac{8\gamma}{l^{\tau+1}}\left(\frac{1}{\nu_1}(l(\omega''-\omega')
+\frac{4\gamma}{l^{\tau}})+1\right)\leq\sum\limits_{l=1}^{+\infty}\frac{8\gamma}{l^{\tau+1}}\mathcal{Q}''l(\omega''-\omega')
\leq \mathcal{Q}'\gamma(\omega''-\omega').
\end{align*}

The argument to prove the upper bounds of $\mathrm{meas}(\mathfrak{R}^2)$ and $\mathrm{meas}(\mathfrak{R}^3)$ is analogous to the one used as above. We will sketch the proof for simplicity. This shows formula \eqref{G11}.

Moreover we get
\begin{align*}
\mathrm{meas}(B_{\gamma}\cap\Omega)=\int_{\epsilon'}^{\epsilon''}\mathrm{meas}(B_{\gamma}(\epsilon)\cap(\omega',\omega''))
~\mathrm{d}\epsilon\stackrel{\eqref{G11}}\geq(1-\mathcal{Q}\gamma)\mathrm{meas}(\Omega).
\end{align*}
This ends the proof of the lemma.
\end{proof}

Theorem \ref{Th1}  follows from Lemma \ref{lem1}, Lemma \ref{lem18} and  Theorem \ref{theo2}.
\begin{proof}[{\bf{Proof of}} {\bf Theorem \ref{Th1}}]
With the help of Theorem \ref{theo2} and the step of Whitney extension (recall \eqref{G1}--\eqref{G6}),  $\tilde{w}(\epsilon,\omega)$  solves $(P)$-equation in \eqref{C1}.
Moreover, $\tilde{w}(\epsilon,\omega)$ is in $C^1(A_\gamma;W\cap \mathcal{H}^s)$. For $\frac{\epsilon}{\gamma^5\omega}\leq\delta_7$ small enough, by virtue of the fact $\|\tilde{w}\|_{s}<r$ (recall \eqref{G6}), Lemma \ref{lem1} presents that $v(\epsilon,w)$ solves $(Q)$-equation in \eqref{C1}. Then, one has that
\begin{align*}
\tilde{u}(\epsilon,\omega):=v(\epsilon,\tilde{w}(\epsilon,\omega))+\tilde{w}(\epsilon,\omega)\in {H}^2_p(0,\pi)\oplus(W\cap \mathcal{H}^s)
\end{align*}
is a solution of equation \eqref{B1}. Meanwhile,  estimates \eqref{C7}--\eqref{C8} can  be obtained  by \eqref{C5} and \eqref{G6}.

 Since $\tilde{u}$ solves
$-(p(x)u_{xx})_{xx}=\epsilon f(t,x,u)-\omega^2\rho(x)u_{tt}$, we obtain
\begin{align*}
-(p(x)\tilde{u}_{xx})_{xx}\in H^{2}(0,\pi),\quad \forall t\in\mathbb{T}.
\end{align*}
Moreover, $\alpha,\beta$ are in $H^4(0,\pi)$, which gives $\rho,p\in H^5(0,\pi)$ according to \eqref{A3}, then, one has
 $\tilde{u}(t,x)\in H^{6}{(0,\pi)}\cap H^{2}_p{(0,\pi)}\hookrightarrow C^5[0,\pi]$ for all $t\in\mathbb{T}$ .
\end{proof}

\section{Invertibility of linearized operators}\label{sec:3}
Let us complete the proof of Lemma \ref{lem1}. More precisely, we have to give the invertibility of operators $\mathcal{L}_{N}(\epsilon,\omega,w)$ (recall \eqref{D3}),
which is the core of any Nash-Moser iteration.

We rewrite $\mathcal{L}_{N}(\epsilon,\omega,w)$  as
\begin{align*}
\mathcal{L}_{N}(\epsilon,\omega,w)[h]
=&\mathfrak{L}_1(\epsilon,\omega,w)[h]+\mathfrak{L}_2(\epsilon,w)[h],\quad \forall h\in W_{N},
\end{align*}
where
\begin{align}
&\mathfrak{L}_1(\epsilon,\omega,w)[h]:=-L_{\omega}h+\epsilon \mathrm{P}_N\Pi_Wf'(t,x,v(\epsilon,\omega,w)+w)h,\nonumber\\
&\mathfrak{L}_2(\epsilon,w)[h]:=\epsilon \mathrm{P}_N\Pi_Wf'(t,x,v(\epsilon,w)+w)\mathrm{D}_{w}v(\epsilon,w)[h].\nonumber
\end{align}
Let  $b(t,x):=f'(t,x,v(\epsilon,\omega,w(t,x))+w(t,x))$. Using \eqref{R11}, $\|w\|_{{s+\sigma}}\leq1$ and Lemma \ref{lem1}, we derive
\begin{align}
\| b\|_{s}&\leq\|b\|_{{s+\sigma}}\leq C,\quad \forall s>{1}/{2},\label{H1}\\
\|b\|_{{s'}}&\leq C(s')(1+\|w\|_{{s'}}),\quad \forall s'\geq s>{1}/{2}.\label{H2}
\end{align}
With the help of decomposing
$$b(t,x)=\sum\limits_{k\in\mathbb{Z}} b_{k}(x)e^{\mathrm{i}k t},\quad h(t,x)=\sum\limits_{1\leq |l|\leq N}h_l(x)e^{\mathrm{i}l t},$$ the operator $\mathfrak{L}_1(\epsilon,\omega,w)$ can be written as
\begin{align*}
\mathfrak{L}_1(\epsilon,\omega,w)[h]=&\sum\limits_{{1\leq|l|\leq N}}\left(\omega^2l^2\rho h_{l}-(p(h_l)'')''\right)e^{\mathrm{i}lt}+\epsilon \mathrm{P}_{N}\Pi_{W}\left(\sum\limits_{k\in\mathbb{Z},{1\leq|l|\leq N}}b_{k-l}h_le^{\mathrm{i}kt}\right)\\
=&\rho\mathfrak{L}_{{1,\mathrm{D}}}[h]-\rho\mathfrak{L}_{{1,\mathrm{ND}}}[h],
\end{align*}
where  $b_0=\Pi_{V}f'(t,x,v(\epsilon,w)+w)$ and
\begin{align}
&\mathfrak{L}_{{1,\mathrm{D}}}[h]=\sum\limits_{1\leq|l|\leq N}\left(\omega^2l^2h_l-\frac{1}{\rho}(ph_l'')'' +\epsilon\frac{b_0}{\rho} h_l\right)e^{\mathrm{i}lt},\nonumber\\
&\mathfrak{L}_{{1,\mathrm{ND}}}[h]=-\frac{\epsilon}{\rho}\sum\limits_{1\leq|l|,|k|\leq N,l\neq k}b_{k-l}h_le^{\mathrm{i}k t}.\nonumber
\end{align}
Let us apply the results of \cite[cf. Theorem 1.2, Proposition 6.2]{MR3346145} to give the  asymptotic formulae of the eigenvalues to  problem \eqref{D4} for $\rho,p$ satisfying \eqref{A3}.
\begin{lemm}\label{lem3}
Denote by $\lambda_{j}(\epsilon, w)$ and $\psi_{j}(\epsilon, w),j\in\mathbb{N}^{+}$ the eigenvalues and the eigenfunctions of problem \eqref{D4} respectively. One has
\begin{align}\label{H4}
\lambda_{1}(\epsilon,w)<\lambda_{2}(\epsilon,w)<\cdots
<\lambda_{j}(\epsilon,w) < \cdots
\end{align}
with $\lambda_{j}(\epsilon,w)\rightarrow+\infty$ as $j\rightarrow+\infty$, and for all $\epsilon\in(\epsilon_1,\epsilon_2)$, $w\in\{W\cap H^s:\|w\|_{s}< r\}$,
\begin{align}\label{H5}
\lambda_j(\epsilon,w)=j^4+2j^2\upsilon_0+\upsilon_1(\epsilon,w)-\varrho_j(\epsilon,w)+\frac{o(1)}{j}\quad\text{as }j\rightarrow+\infty.
\end{align}
Here,
\begin{align}\label{H3}
\upsilon_0=&\mathfrak{d}(\pi)-\mathfrak{d}(0)+\frac{1}{\pi}\int^{0}_{\pi}\frac{\mathfrak{x}(x)}{\zeta(x)}\mathrm{d}x,\nonumber\\
\upsilon_1(\epsilon,w)=&\mathfrak{e}(\pi)-\mathfrak{e}(0)+\frac{1}{\pi}\int^{\pi}_{0}\mathfrak{g}(x)\zeta(x)\mathrm{d}x
+\frac{\varrho^2_0}{2}-\epsilon\frac{1}{\pi}\int^{\pi}_{0}\Pi_{V}f'(v(\epsilon,w)+w)(x)\zeta(x)\mathrm{d}x,\nonumber\\
\varrho_j(\epsilon,w)=&\frac{1}{\pi}\int^{\pi}_{0}\left(-\epsilon \Pi_{V}f'(v(\epsilon,w)+w)(x)\zeta(x)+\frac{\alpha'''(x)-\beta'''(x)}{4\zeta^3(x)}\right)\cos\left(2j\int^{x}_0\zeta(z)\mathrm{d}z\right)\mathrm{d}x,
\end{align}
where
\begin{align}
\mathfrak{d}=&\frac{3\alpha+5\beta}{2\xi},\quad \zeta=\left({\rho}/{p}\right)^{1/4},\quad \mathfrak{x}=\frac{5\alpha^2+5\beta^2+6\alpha\beta}{4}\geq\frac{\alpha^2+\beta^2}{2}\geq0,\label{H0}\\
\mathfrak{e}=&\frac{1}{\zeta^3}\left(\frac{2\mathfrak{z}^3}{3}-\frac{\eta^3_{-}}{2}-2\eta\eta_{+}-(\mathfrak{z}-\eta_{-})\eta_{-}\mathfrak{z}
+(\mathfrak{z}-\eta_{-})'\mathfrak{z}-(\alpha\eta_{-})'-\frac{(\eta_{-})''}{4}\right),\nonumber\\
\mathfrak{g}=&\frac{1}{8\zeta^4}\left(((\eta_{-})'-\eta^2_{-}-2\mathfrak{x})^2-8((\eta_{+})'-2\eta)^2\right),\nonumber\\
\mathfrak{z}=&\frac{\alpha+3\beta}{2},\quad\eta=\eta_{+}\eta_{-},\quad \eta_{\pm}=\beta\pm\alpha.\nonumber
\end{align}
 And $\psi_{j}(\epsilon,w)$ form an orthogonal basis of $L^2(0,\pi)$ with the scalar product
\[(y,z)_{L^2_{\rho}}:=\int_{0}^{\pi}\rho yz\mathrm{d}x.\]
Moreover, define an equivalent scalar product $(\cdot,\cdot)_{\epsilon,w}$ on ${H}^2_p(0,\pi)$ by
\begin{align*}
(y,z)_{\epsilon,w}:=\int_{0}^{\pi}py''z''-{\epsilon}\Pi_{V}f'(v(\epsilon,\omega,w)+w))yz+M\rho yz\mathrm{d}t
\end{align*}
with
\begin{align}\label{H6}
L_1\|y\|_{H^2}\leq\|y\|_{\epsilon,w}\leq L_2\|y\|_{H^2},\quad\forall y\in{H}_p^2(0,\pi)
\end{align}
for some constants $L_1$, $L_2>0$. The eigenfunctions $\psi_{j}(\epsilon,w)$ are also an orthogonal basis of ${H}^2_p(0,\pi)$
 with respect to the scalar product $(\cdot,\cdot)_{\epsilon,w}$ and one has that for all $ y=\sum_{j\geq1}\hat{y}_j\psi_{j}(\epsilon,w)$,
\begin{align}\label{H7}
\|y\|^2_{L^2_{\rho}}=\sum\limits_{j\geq1}(\hat{y}_j)^2,\quad
\|y\|^2_{\epsilon,w}=\sum\limits_{j\geq1}(\lambda_j(\epsilon,w)+M)(\hat{y}_j)^2,
\end{align}
where $\lambda_{j}{(\epsilon,w)}+M>0$ for $M>0$ large enough.
\end{lemm}
\begin{proof}
Since the asymptotic formulae of the eigenvalues can be get similar by Theorem 1.2 and Proposition 6.2 of \cite {MR3346145},
we just need to check \eqref{H6}--\eqref{H7}.

Using Poincar\'{e} inequality yields $\|y'\|_{L^2(0,\pi)}\leq C\|y''\|_{L^2(0,\pi)}$. A simple calculation gives  \eqref{H6}. Moreover, one has
\begin{align*}
-(p\psi''_j(\epsilon,w))''-\epsilon\Pi_{V}f'(t,x,v(\epsilon,w)+w)\psi_j(\epsilon,w)=\lambda_j(\epsilon,w)\psi_j(\epsilon,w).
\end{align*}
Multiplying above equality  by $\psi_{j'}(\epsilon,w)$ and integrating by parts yields
\begin{align*}
(\psi_j,\psi_{j'})_{\epsilon,w}=\delta_{j,j'}\lambda_j(\epsilon,w),
\end{align*}
which implies \eqref{H7} for all $ y=\sum_{j\geq1}\hat{y}_j\psi_{j}(\epsilon,w)$.
\end{proof}
Formula \eqref{H6} shows the equivalent norm of the $s$-norm restricted to $W\cap \mathcal{H}^s$, i.e.,
\begin{align}\label{H8}
&L^2_1\|w\|^2_{s}\leq\sum\limits_{|l|\geq1,j\geq1}(\lambda_j(\epsilon,w)+M)(\hat{w}_{l,j})^2(1+l^{2s})\leq L^2_2\|w\|^2_{s}
\end{align}
for all $w=\sum_{|l|\geq1,j\geq1}\hat{w}_{l,j}\psi_{j}(\epsilon,w)e^{\mathrm{i}lt},$ where $L_i,i=1,2$ are given in \eqref{H6}. Moreover, applying Lemma \ref{lem3} yields that for $h_{l}=\sum^{+\infty}_{j=1}\hat{h}_{l,j}\psi_j(\epsilon,w),$
\[\omega^2l^2h_l-\frac{1}{\rho}(ph_l'')'' +\epsilon\frac{b_0}{\rho} h_l=\sum\limits^{+\infty}_{j=1}(\omega^2l^2-\lambda_j(\epsilon,w))\hat{h}_{l,j}\psi_j(\epsilon,w).\]
Then, $\mathfrak{L}_{{1,\mathrm{D}}}$  is a diagonal operator on $W_{N}$.

Define $|\mathfrak{L}_{{1,\mathrm{D}}}|^{\frac{1}{2}}$ by
\begin{align*}
|\mathfrak{L}_{{1,\mathrm{D}}}|^{\frac{1}{2}}h=\sum\limits_{1\leq|l|\leq N,j\geq1}\sqrt{|\omega^2l^2-\lambda_j(\epsilon,w)|} \hat{h}_{l,j}\psi_j(\epsilon,w)e^{\mathrm{i}lt},\quad \forall h\in W_{N}.
\end{align*}
If  $\omega^2l^2-\lambda_j(\epsilon,w)\neq0$, $\forall1\leq|l|\leq N,~\forall j\geq1$, its invertibility is
\begin{align*}
|\mathfrak{L}_{{1,\mathrm{D}}}|^{-\frac{1}{2}}{h}:=\sum\limits_{1\leq|l|\leq N,j\geq1}\frac{1}{\sqrt{|\omega^2l^2-\lambda_j(\epsilon,w)|}}\hat{{h}}_{l,j}
\psi_j(\epsilon,w)e^{\mathrm{i}lt}.
\end{align*}
Hence we can rewrite $\mathcal{L}_{N}(\epsilon,\omega,w)$ as
\begin{align*}
\mathcal{L}_{N}(\epsilon,\omega,w)=\rho|\mathfrak{L}_{1,{\mathrm{D}}}|^{\frac{1}{2}}(|\mathfrak{L}_{1,{\mathrm{D}}}|^{-\frac{1}{2}}
\mathfrak{L}_{1,{\mathrm{D}}}|\mathfrak{L}_{1,{\mathrm{D}}}|^{-\frac{1}{2}}-{R}_1-{R}_2)|\mathfrak{L}_{1,{\mathrm{D}}}|^{\frac{1}{2}},
\end{align*}
where
\begin{equation}\label{H9}
{R}_1=|\mathfrak{L}_{1,{\mathrm{D}}}|^{-\frac{1}{2}}
\mathfrak{L}_{1,{\mathrm{ND}}}|\mathfrak{L}_{1,{\mathrm{D}}}|^{-\frac{1}{2}},\quad
{R}_2=-|\mathfrak{L}_{1,{\mathrm{D}}}|^{-\frac{1}{2}}
\left(1/\rho\mathfrak{L}_2\right)|\mathfrak{L}_{1,{\mathrm{D}}}|^{-\frac{1}{2}}.
\end{equation}
Consequently, it follows from the definitions of $\mathfrak{L}_{1,\mathrm{D}}$, $|\mathfrak{L}_{1,{\mathrm{D}}}|^{-\frac{1}{2}}$ that
\begin{align*}
(|\mathfrak{L}_{1,{\mathrm{D}}}|^{-\frac{1}{2}}
\mathfrak{L}_{1,{\mathrm{D}}}|\mathfrak{L}_{1,{\mathrm{D}}}|^{-\frac{1}{2}}){h}=\sum\limits_{1\leq|l|\leq N,j\geq1}\mathrm{sign}(\omega^2l^2-\lambda_j(\epsilon,w))
\hat{h}_{l,j}\psi_{j}(\epsilon,w)e^{\mathrm{i}lt},\quad \forall{h}\in{W}_{N}.
\end{align*}
Then, it is invertible with
\begin{align}
\|(|\mathfrak{L}_{1,{\mathrm{D}}}|^{-\frac{1}{2}}
\mathfrak{L}_{1,{\mathrm{D}}}|\mathfrak{L}_{1,{\mathrm{D}}}|^{-\frac{1}{2}})^{-1}{h}\|_{s}
\stackrel{\eqref{H8}}{\leq}
 \frac{L_2}{L_1}\|h\|_{s}\label{H10},\quad\forall h\in W_{N},\forall s\geq0.
\end{align}
Therefore, $\mathcal{L}_{N}(\epsilon,\omega,w)$ may be reduced to
\begin{align}\label{H11}
\mathcal{L}_{N}(\epsilon,\omega,w)=\rho|\mathfrak{L}_{1,{\mathrm{D}}}|^{\frac{1}{2}}(|\mathfrak{L}_{1,{\mathrm{D}}}|^{-\frac{1}{2}}
\mathfrak{L}_{1,{\mathrm{D}}}|\mathfrak{L}_{1,{\mathrm{D}}}|^{-\frac{1}{2}})(\mathrm{Id}-\mathcal{R})
|\mathfrak{L}_{1,{\mathrm{D}}}|^{\frac{1}{2}},
\end{align}
where $\mathcal{R}=\mathcal{R}_1+\mathcal{R}_2$ with
\begin{align*}
\mathcal{R}_1=(|\mathfrak{L}_{1,{\mathrm{D}}}|^{-\frac{1}{2}}
\mathfrak{L}_{1,{\mathrm{D}}}|\mathfrak{L}_{1,{\mathrm{D}}}|^{-\frac{1}{2}})^{-1}R_1,\quad
\mathcal{R}_2=(|\mathfrak{L}_{1,{\mathrm{D}}}|^{-\frac{1}{2}}
\mathfrak{L}_{1,{\mathrm{D}}}|\mathfrak{L}_{1,{\mathrm{D}}}|^{-\frac{1}{2}})^{-1}R_2.
\end{align*}

To verify the invertibility of the operator $\mathrm{Id}-\mathcal{R}$ in \eqref{H11}, we have to suppose some non-resonance conditions.

 For $\tau\in(1,2)$, assume  the following ``Melnikov's'' non-resonance conditions:
\begin{align}\label{H12}
|\omega l-\mu_j(\epsilon,w)|>\frac{\gamma}{l^{\tau}},\quad\forall 1\leq l\leq N,~\forall j\geq 1.
\end{align}
It follows from the definition of $\mu_{j}(\epsilon,w)$ (recall \eqref{D5}) that
\begin{align}\label{H13}
|\omega^2l^2-\lambda_j(\epsilon,w)|=|\omega l-\mu_j(\epsilon,w)||\omega l+\mu_j(\epsilon,w)|>\frac{\gamma\omega}{l^{\tau-1}},\quad\forall 1\leq l\leq N,~\forall j\geq 1.
\end{align}
Furthermore denote
\begin{equation}\label{H14}
\omega_{l}:=\min_{j\geq1}|\omega^2l^2-\lambda_j(\epsilon,w)|=|\omega^2l^2-\lambda_{j^*}(\epsilon,w)|, \quad\forall1\leq|l|\leq N.
\end{equation}
It is clear that $\omega_l=\omega_{-l}$ for all $1\leq l\leq N$.

\begin{lemm}\label{lem4}
Given assumption \eqref{H12}, for all $s\geq0,{h}\in{W}_{N}$, the operator $|\mathfrak{L}_{{1,\mathrm{D}}}|^{-\frac{1}{2}}$ is invertible with
\begin{align}\label{H15}
\||\mathfrak{L}_{{1,\mathrm{D}}}|^{-\frac{1}{2}}{h}\|_{{s}}\leq&\frac{\sqrt{2}L_2}{\sqrt{\gamma\omega}L_1}\|{h}\|_{{s+\frac{\tau-1}{2}}},\\
\||\mathfrak{L}_{{1,\mathrm{D}}}|^{-\frac{1}{2}}{h}\|_{{s}}\leq&\frac{\sqrt{2}{L_2}}{\sqrt{\gamma\omega}L_1}N^{\frac{\tau-1}{2}}\|{h}\|_{{s}}.\label{H16}
\end{align}
\end{lemm}
\begin{proof}
Since  $|l|^{\tau-1}(1+l^{2s})<2(1+|l|^{2s+\tau-1})$ for all $1\leq|l|\leq N$, by \eqref{H8}, \eqref{H13}--\eqref{H14}, one has
\begin{align*}
\||\mathfrak{L}_{{1,\mathrm{D}}}|^{-\frac{1}{2}}{h}\|^2_{{s}}
\leq&\frac{1}{L^2_1}\sum\limits_{1\leq|l|\leq N,j\geq1}\frac{\lambda_j(\epsilon,w)+M}{\omega_l}(\hat{{h}}_{l,j})^2(1+l^{2s})\\
\leq&\frac{2}{\gamma\omega L^2_1}\sum\limits_{
1\leq|l|\leq N,j\geq1}(\lambda_j(\epsilon,w)+M)(\hat{h}^{\pm}_{{l,j}})^2(1+|l|^{2s+\tau-1})\\
\leq &\frac{2L^2_2}{\gamma\omega L^2_1}\|h\|^2_{{s+\frac{\tau-1}{2}}}\leq \frac{2L^2_2N^{\tau-1}}{\gamma \omega L^2_1}\|h\|^2_{{s}}.
\end{align*}
This shows the conclusion of the lemma.
\end{proof}

The next step is to verify the upper bounds of $\left\|\mathcal{R}_ih\right\|_{s'},i=1,2$ for all $s'\geq s>{1}/{2}$.
Moreover, for $\tau\in(1,2)$, we also assume ``Melnikov's'' non-resonance conditions:
\begin{align}\label{H17}
\left|\omega {l}-{j}\right|>\frac{\gamma}{l^{\tau}},\quad\forall 1\leq l\leq N,\forall j\geq 1.
\end{align}
(In fact, condition \eqref{H17} will be used for the proof of {\bf(F2)}.)
\begin{lemm}\label{lem5}
Supposed that \eqref{H12} and \eqref{H17} hold, if $\|w\|_{s+\sigma}\leq1$ with $\sigma$ being seen in \eqref{D7},  then
 there exists some constant  $L(s')>0$ such that for all $s'\geq s>{1}/{2}$,
\begin{equation}\label{H18}
\left\|\mathcal{R}_1h\right\|_{s'}\leq\frac{\epsilon L(s')}{2\gamma^3\omega}\left(\|h\|_{s'}+\|w\|_{s'+\sigma}\|h\|_s\right),\quad\forall h\in W_{N}.
\end{equation}
\end{lemm}
\begin{proof}
Let us first claim the following:\\
{\bf (F2)}: Fix $\tau\in(1,2),\gamma\in(0,1)$ and $\omega>\gamma$. Provided \eqref{H12} and \eqref{H17} hold for all $|l|,|k|\in\{1,\cdots,N\}$ with $l\neq k$,  there is some constant $\tilde{L}>0$ such that
\begin{align*}
\omega_l\omega_k\geq\frac{\tilde{L}^2\gamma^6\omega^2}{|l-k|^{2\sigma}}.
\end{align*}

Using  formula \eqref{H9} and the definitions of $\mathfrak{L}_{1,{\mathrm{ND}}},|\mathfrak{L}_{{1,\mathrm{D}}}|^{-\frac{1}{2}}$ yields
\begin{align*}
R_1h
&=|\mathfrak{L}_{1,{\mathrm{D}}}|^{-\frac{1}{2}}\mathfrak{L}_{1,{\mathrm{ND}}}\left(\sum\limits_{1\leq|l|\leq N,j\geq1}\frac{\hat{{h}}_{l,j}}{\sqrt{|\omega^2l^2-\lambda_j(\epsilon,w)|}}\psi_j(\epsilon,w)e^{\mathrm{i}lt}\right)\\
&=-\epsilon|\mathfrak{L}_{1,{\mathrm{D}}}|^{-\frac{1}{2}}\left(\sum\limits_{\stackrel{1\leq|l|,|k|\leq N}{l\neq k,j\geq1}}\frac{\hat{h}_{l,j}}{\sqrt{|\omega^2l^2-\lambda_j(\epsilon,w)|}}\frac{b_{k-l}}{\rho}
\psi_j(\epsilon,w)e^{\mathrm{i}k t}\right)\\
&=-\epsilon\sum\limits_{\stackrel{1\leq|l|,|k|\leq N}{l\neq k,j\geq1}}\frac{\hat{h}_{l,j}}{\sqrt{|\omega^2k^2-\lambda_j(\epsilon,w)|}
\sqrt{|\omega^2l^2-\lambda_j(\epsilon,w)|}}\frac{b_{k-l}}{\rho}\psi_j(\epsilon,w)e^{\mathrm{i}k t}.
\end{align*}
This leads to
\begin{align*}
(R_1h)_{k}=-\epsilon\sum\limits_{\stackrel{1\leq|l|\leq N}{l\neq k,j\geq 1}}\frac{\hat{h}_{l,j}}{\sqrt{|\omega^2k^2-\lambda_j(\epsilon,w)|}
\sqrt{|\omega^2l^2-\lambda_j(\epsilon,w)|}}\frac{b_{k-l}}{\rho}\psi_j(\epsilon,w).
\end{align*}
Combining this with \eqref{H6}--\eqref{H7} and {\bf(F2)}, we can obtain
\begin{align}
\|(R_1h)_k\|_{H^2}\leq&\frac{\epsilon L_2}{L_1}\sum\limits_{1\leq|l|\leq N,l\neq k}\frac{1}{\sqrt{\omega_l\omega_k}}\left\|{b_{k-l}}/{\rho}\right\|_{H^2}\|h_l\|_{H^2}\nonumber\\
{\leq}&\frac{\epsilon L_2}{\gamma^3\omega \tilde{L}L_1}\sum\limits_{1\leq|l|\leq N,l\neq k}\left\|{b_{k-l}}/{\rho}\right\|_{H^2}|k-l|^\sigma\|h_l\|_{H^2}.\label{H20}
\end{align}

Let us define
\begin{align*}
&{\mathfrak{r}}(x):=\sum\limits_{1\leq |l|,|k|\leq N}\left\|{b_{k-l}}/{\rho}\right\|_{H^2}|k-l|^\sigma\|h_l\|_{H^2}e^{\mathrm{i}kt}\quad\text{with }b_0=0,\\ &\mathfrak{p}(x):=\sum_{l\in\mathbb{Z}}\left\|{b_{l}}/{\rho}\right\|_{H^2}|l|^{\sigma} e^{\mathrm{i}lt},\quad
\mathfrak{q}(x):=\sum_{1\leq|l|\leq N}\|h_l\|_{H^2}e^{\mathrm{i}lt}.
\end{align*}
Clearly, ${\mathfrak{r}}=\mathrm{P}_{N}(\mathfrak{p}\mathfrak{q})$ and  for all $s'\geq s>\frac{1}{2}$,
\begin{equation*}
\|\mathfrak{p}\|_{s'}\leq \sqrt{2}\|1/\rho\|_{H^2}\|b\|_{s'+\sigma}\stackrel{\eqref{H2}}{\leq}C'(s')(1+\|w\|_{s'+\sigma}),
\quad
\|\mathfrak{q}\|_{s'}=\|h\|_{s'}.
\end{equation*}
Hence we can deduce that for $\|w\|_{s+\sigma}\leq1$,
\begin{align*}
\|R_1h\|_{s'}\stackrel{\eqref{H20}}{\leq}&\frac{\epsilon L_2}{\gamma^3\omega \tilde{L}L_1}\|\mathfrak{r}\|_{s'}
\stackrel{\eqref{R2}}{\leq}\frac{\epsilon L_2}{\gamma^3\omega \tilde{L}L_1}C(s')(\|\mathfrak{p}\|_{s'}\|\mathfrak{q}\|_{s}+\|\mathfrak{p}\|_{s}\|\mathfrak{q}\|_{s'})\\
\leq&\frac{\epsilon L_2 C''(s')}{\gamma^3\omega \tilde{L}L_1}(\|w\|_{s'+\sigma}\|h\|_s+\|h\|_{s'}).
\end{align*}
Combining above inequality with \eqref{H10} completes the proof of the lemma if we take $L(s')/2=\frac{L^2_2 C''(s')}{ \tilde{L}L^2_1}.$
\end{proof}

\begin{lemm}\label{lem6}
Under the non-resonance conditions  \eqref{H12}, for $\|w\|_{s+\sigma}\leq1$ with $\sigma$ being seen in \eqref{D7}, we have
\begin{align}\label{H25}
\left\|\mathcal{R}_2h\right\|_{s'}\leq\frac{\epsilon L(s')}{2\gamma\omega}\left(\|h\|_{s'}+\|w\|_{s'+\sigma}\|h\|_s\right),\quad\forall h\in W_{N},\forall s'\geq s>{1}/{2}.
\end{align}
\end{lemm}
\begin{proof}
It is straightforward that $\sigma>\tau-1$ due to the fact  $\tau\in(1,2)$. Moreover, it follows from that Lemma \ref{lem1} that
$$\mathrm{D}_{w}v(\epsilon,w)[|\mathfrak{L}_{1,{\mathrm{D}}}|^{-\frac{1}{2}}h]\in {H}^2_p(0,\pi).$$
Consequently, by virtue of \eqref{H1}--\eqref{H2}, we have
\begin{align*}
\|R_2h\|_{s'}
&\stackrel{\eqref{H15}}{\leq}\frac{\epsilon\sqrt{2}L_2}{\sqrt{\gamma\omega}L_1}\|1/\rho\|_{H^2}\|b(t,x)\mathrm{D}_wv(\epsilon,w)[|\mathscr{L}_{1,{\mathrm{D}}}|^{-\frac{1}{2}}h]
\|_{s'+\frac{\tau-1}{2}}\\
&\stackrel{\eqref{R2}}{\leq}\frac{\epsilon \sqrt{2} L_2}{\sqrt{\gamma\omega}L_1}\|1/\rho\|_{H^2}C'(s')(\|b\|_{s'+\sigma}\||\mathfrak{L}_{1,{\mathrm{D}}}|^{-\frac{1}{2}}h\|_{s-\frac{\tau-1}{2}}
+\|b\|_{s+\sigma}\||\mathfrak{L}_{1,{\mathrm{D}}}|^{-\frac{1}{2}}h\|_{s-\frac{\tau-1}{2}})\\
&\stackrel{\eqref{H15}}{\leq}\frac{2\epsilon L^2_2C''(s')}{{\gamma\omega}L^2_1}(\|w\|_{s'+\sigma}\|h\|_s+\|h\|_{s'}).
\end{align*}
Hence, using \eqref{H10} yields the conclusion of the lemma if $L(s')/2=\frac{2L^3_2 C''(s')}{ L^3_1}$.
\end{proof}

\begin{lemm}\label{lem7}
Given \eqref{H12} and \eqref{H17}, if $\|w\|_{s+\sigma}\leq1$ and ${\epsilon L(s')}/{(\gamma^3\omega)}\leq \mathrm{c}$ is small enough, one has that the operator $(\mathrm{Id}-\mathcal{R})$ is invertible with
\begin{align}\label{H26}
\|(\mathrm{Id}-\mathcal{R})^{-1}h\|_{s'}\leq 2(\|h\|_{s'}+\|w\|_{s'+\sigma}\|h\|_s),\quad\forall h\in W_{N},\forall s'\geq s>{1}/{2}.
\end{align}
\end{lemm}
\begin{proof}
If ${\epsilon L}/{(\gamma^3\omega)}\leq \mathrm{c}$ small enough,  it follows from  Lemmata \ref{lem5}--\ref{lem6} that for $\|w\|_{s+\sigma}\leq1$, one has
\begin{align*}
\|\mathcal{R}h\|_{s}\leq{\epsilon L(\gamma^3\omega)^{-1}}\|h\|_{s}\leq{1}/{2}\|h\|_s\quad \text{with } L\leq L(s').
\end{align*}
Then, by Neumann series, the operator $(\mathrm{Id}-\mathcal{R})$ is invertible in $(W_{N},\|\cdot\|_{s})$.

 Next, let us  claim the following:\\
%
{\bf (F3)}: If $\|w\|_{s+\sigma}\leq1$, then
\begin{equation}\label{H24}
\|\mathcal{R}^\mathfrak{n}h\|_{s'}\leq(\epsilon L(s')(\gamma^3\omega)^{-1})^\mathfrak{n}(\|h\|_{s'}+\mathfrak{n}\|w\|_{s'+\sigma}\|h\|_s),\quad\forall h\in W_{N},  \forall \mathfrak{n}\in\mathbb{N}^{+}.
\end{equation}

Hence, for $\epsilon L(s')(\gamma^3\omega)^{-1}\leq {\mathrm{{c}}}(s')\leq \mathrm{c}$ small enough,  above inequality reads
\begin{align*}
\|(\mathrm{Id}-\mathcal{R})^{-1}h\|_{s'}=&\|(\mathrm{Id}+\sum\limits_{\mathfrak{n}\in\mathbb{N}^{+}}\mathcal{R}^\mathfrak{n})h\|_{s'}\leq\|h\|_{s'}+\sum\limits_{\mathfrak{n}\in\mathbb{N}^{+}}\|\mathcal{R}^\mathfrak{n}h\|_{s'}\\
\leq&\|h\|_{s'}+\sum\limits_{\mathfrak{n}\in\mathbb{N}^{+}}(\epsilon L(s')(\gamma^3\omega)^{-1})^\mathfrak{n}(\|h\|_{s'}+\mathfrak{n}\|w\|_{s'+\sigma}\|h\|_s)\\
\leq&2\|h\|_{s'}+2\|w\|_{s'+\sigma}\|h\|_s
\end{align*}
for all $h\in W_{N}$ and $s'\geq s>\frac{1}{2}$.

 Let us prove {\bf (F3)} by induction. Formulae \eqref{H18} and \eqref{H25} show that for $\mathfrak{n}=1$,
\begin{align*}
\left\|\mathcal{R}h\right\|_{s'}\leq{\epsilon L(s')(\gamma^3\omega)^{-1}}\left(\|h\|_{s'}+\|w\|_{s'+\sigma}\|h\|_s\right).
\end{align*}
Assume that \eqref{H24} holds for $\mathfrak{n}=\mathfrak{l}$ with $\mathfrak{l}\in\{\mathfrak{l}\in\mathbb{N}^+:\mathfrak{l}\geq2\}$. Let us check that \eqref{H24}  holds for $\mathfrak{n}=\mathfrak{l}+1$. Based on the assumption for $\mathfrak{n}=\mathfrak{l}$,
we can obtain
\begin{align*}
\|\mathcal{R}^{\mathfrak{l}+1}h\|_{s'}=&\|\mathcal{R}^\mathfrak{l}(\mathcal{R}h)\|_{s'}{\leq}
(\epsilon L(s')(\gamma^3\omega)^{-1})^\mathfrak{l}(\|\mathcal{R}h\|_{s'}+\mathfrak{l}\|w\|_{s'+\sigma}\|\mathcal{R}h\|_s)\\
\leq&(\epsilon L(s')(\gamma^3\omega)^{-1})^l\left(\epsilon L(s')(\gamma^3\omega)^{-1}\|h\|_{s'}+(\mathfrak{l}\epsilon L(\gamma^3\omega)^{-1}+\epsilon L(s')(\gamma^3\omega)^{-1})\|w\|_{s'+\sigma}\|h\|_s\right)\\
\leq&(\epsilon L(s')(\gamma^3\omega)^{-1})^{\mathfrak{l}+1}(\|h\|_{s'}+(\mathfrak{l}+1)\|w\|_{s'+\sigma}\|h\|_s),
\end{align*}
which completes the proof of {\bf (F3)}.
\end{proof}
Let us complete the proof of Lemma \ref{lem1}.

\begin{proof}[{\bf{Proof of Lemma \ref{lem1}}}]
It follows from formulae \eqref{H16}, \eqref{H26} and \eqref{H10} that
\begin{align*}
\|\mathcal{L}^{-1}_N(\epsilon,\omega,w)h\|_{s'}\stackrel{\eqref{H11}}{\leq}
\frac{K(s')}{\gamma\omega}N^{\tau-1}(\|h\|_{s'}+\|w\|_{s'+\sigma}\|h\|_{s}),
\end{align*}
where $K(s'):=\frac{8L^3_2}{{L^3_1}}\|1/\rho\|_{H^2}$.
In particular, we get that for $\|w\|_{s+\sigma}\leq1$,
\begin{align*}
\|\mathcal{L}^{-1}_N(\epsilon,\omega,w)h\|_{s}\leq\frac{K}{\gamma \omega}N^{\tau-1}\|h\|_{s}.
\end{align*}
\end{proof}

Now we are devoted to checking {\bf(F2)}. Since $\alpha,\beta$ belong to $H^4(0,\pi)$, which shows $\rho,p\in H^5(0,\pi)$ according to \eqref{A3}. If $\Pi_{V}f'(t,x,v(\epsilon,w)+w)\in H^2_{p}(0,\pi)$,  one has
\begin{align*}
\mathfrak{m}(\epsilon,w)\in H^1(0,\pi),
\end{align*}
where $\mathfrak{m}(\epsilon,w):=-\epsilon \Pi_{V}f'(v(\epsilon,w)+w)\zeta+\frac{\alpha'''-\beta'''}{4\zeta^3},$
and $\zeta$ is defined in \eqref{H0}.
Then integrating by parts implies
\begin{align*}
|\varrho_j(\epsilon,w)|=\left|\frac{1}{\pi}\int^{\pi}_{0}\mathfrak{m}(\epsilon,w)(x)\cos\left(2j\int^{x}_0\zeta(z)\mathrm{d}z\right)\mathrm{d}x\right|\leq\frac{\|\mathfrak{m}/\zeta\|_{H^1}}{j},
\end{align*}
where $\varrho_{j}(\epsilon,w)$ is defined in \eqref{H3}. Thus it follows from \eqref{H5} that for $M>0$ large enough,
\begin{align}\label{H21}
\lambda_j(\epsilon,w)=j^4+2j^2\upsilon_0+\upsilon_1(\epsilon,w)+\frac{r(\epsilon,w)}{j}\quad \text{with } |r(\epsilon,w)|\leq M\quad\text{as }j\rightarrow+\infty.
\end{align}
Using \eqref{D5}, \eqref{H21} and Taylor expansion yields that there exists $\mathcal{J}_0>\max\{2|\upsilon_0|,1\}>0 $ large enough such that for all $j>\mathcal{J}_0$,
\begin{equation}\label{H22}
\left|\mu_j(\epsilon,w)-(j^2+\upsilon_0)\right|\leq \frac{M}{j^2}
\end{equation}
 for $M>0$ large enough. Moreover if $\omega^2l^2-\lambda_{\mathcal{J}_0+1}(\epsilon,w)>0$, then $j^*\geq\mathcal{J}_0+1$, where $j^*$ is seen in \eqref{H14}. Hence there exists $\mathcal{J}_1:=\mathcal{J}_1(\mathcal {J}_0)>0$ such that for every $l>\mathcal{J}_1/\omega$,
\begin{equation}\label{H23}
j^*\geq \tilde{M}\sqrt{\omega l}.
\end{equation}

\begin{proof}[{\bf{Proof of}} {\bf(F2)}]
Let  $l,k\geq 1$ with $l\neq k$ and set
$\omega_{l}=\min_{j\geq1}|\omega^2l^2-\lambda_j(\epsilon,w)|=|\omega^2l^2-\lambda_{j^*}(\epsilon,w)|$ and $\omega_{k}=|\omega^2k^2-\lambda_{i^*}(\epsilon,w)$.

Case 1: $2|k-l|>(\max{\{k,l\}})^{\varsigma}$, where $\varsigma=(2-\tau)/\tau\in(0,1)$ by $\tau\in(1,2)$. It follows form \eqref{H13} that
\begin{align*}
\omega_{l}\omega_{k}>\frac{(\gamma\omega)^2}{(kl)^{\tau-1}}\geq\frac{(\gamma\omega)^2}{(\max{\{k,l\}})^{2(\tau-1)}}
>\frac{(\gamma\omega)^2}{2^{2(\tau-1)/\varsigma}|k-l|^{2(\tau-1)/\varsigma}}.
\end{align*}

Case 2: $0<2|k-l|\leq(\max{\{k,l\}})^{\varsigma}$. Then either $k>l$ or $l>k$ holds. Using the fact $\varsigma\in(0,1)$, in the first case $2l>k$ and in the latter $2k>l$, i.e.,
\begin{equation*}
{k}/{2}<l<2k.
\end{equation*}

$\mathrm{(i)}$ If $\lambda_{j^*}(\epsilon,w),\lambda_{i^*}(\epsilon,w)<0$, then $\omega_l\geq\omega^2l^2,\omega_k\geq\omega^2k^2$, which leads to
\begin{align*}
\omega_l\omega_k\geq\omega^2>\gamma^2\omega^2.
\end{align*}

$\mathrm{(ii)}$ We consider either $\lambda_{j^*}(\epsilon,w)<0$ or $\lambda_{i^*}(\epsilon,w)<0$. Then in the first case
\begin{align*}
\omega_l\omega_k\stackrel{\eqref{H13}}{>}\omega^2 l^2\frac{\gamma\omega}{k^{\tau-1}}>2^{1-\tau}\gamma\omega^3>2^{1-\tau}\gamma^2\omega^2,
\end{align*}
and in the latter
\begin{align*}
\omega_l\omega_k\stackrel{\eqref{H13}}{>}\frac{\gamma\omega}{l^{\tau-1}}\omega^2 k^2>2^{1-\tau}\gamma\omega^3>2^{1-\tau}\gamma^2\omega^2.
\end{align*}

$\mathrm{(iii)}$ We restrict our attention to the case $\lambda_{j^*}(\epsilon,w)>0,\lambda_{i^*}(\epsilon,w)>0$. Suppose $\max{\{k,l\}}=k\geq k_\star$ with $k_{\star}:=\max\left\{\frac{2\mathcal{J}_1}{\omega},\left(\frac{6M}{\tilde{M}^2\gamma\omega}\right)^{\frac{1}{1-\varsigma\tau}}\right\}$. It follows from \eqref{H17}, \eqref{H22}--\eqref{H23} and $\tau\in(1,2)$ that
\begin{align*}
&\Big|\Big(\omega l-\mu_{j^*}(\epsilon,w)\Big)
-\Big(\omega k-\mu_{i^*}(\epsilon,w)\Big)\Big|\\
\geq&\left|\omega(l-k)-((i^*)^2-(j^*)^2)\right|-\left|\mu_{j^*}(\epsilon,w)-((j^*)^2+\upsilon_0)\right|
-\left|\mu_{i^*}(\epsilon,w)-((i^*)^2+\upsilon_0)\right|\\
>&\frac{\gamma}{|l-k|^{\tau}}-\frac{M}{\tilde{M}^2\omega l}-\frac{M}{\tilde{M}^2\omega k}
\geq \frac{2^\tau \gamma}{k^{\varsigma\tau}}-\frac{3M}{\tilde{M}^2\omega k}\\
>&\frac{\gamma}{2k^{\varsigma\tau}}+\frac{\gamma}{k^{\varsigma\tau}}+\frac{\gamma}{2k^{\varsigma\tau}}-\frac{3M}{\tilde{M}^2\omega k}\stackrel{k<2l }{>}\frac{1}{2}\left(\frac{\gamma}{k^{\varsigma\tau}}+\frac{\gamma}{l^{\varsigma\tau}}\right),
\end{align*}
which implies
\begin{equation*}
\mathrm{either}\quad\left|\omega k-\mu_{i^*}(\epsilon,w)\right|>\frac{\gamma}{2k^{\varsigma\tau}}\quad\mathrm{or}\quad\left|\omega l-\mu_{j^*}(\epsilon,w)\right|>\frac{\gamma}{2l^{\varsigma\tau}}~ \mathrm{holds}.
\end{equation*}
 The same conclusion is reached if $\max{\{k,l\}}=l\geq k_\star$. Without loss of generality, we suppose $|\omega k-\mu_{i^*}(\epsilon,w)|>\frac{\gamma}{2k^{\varsigma\tau}}$. Then
\begin{align*}
\omega_k=|\omega^2 k^2-\lambda_{i^*}(\epsilon,w)|=|\omega k-\mu_{i^*}(\epsilon,w)||\omega k+\mu_{i^*}(\epsilon,w)|> \frac{\gamma\omega}{2}k^{1-\varsigma\tau},
\end{align*}
which shows
\begin{align*}
\omega_l\omega_k{\geq} \frac{\gamma\omega}{l^{\tau-1}} \frac{\gamma\omega}{2}k^{1-\varsigma\tau}
>\frac{(\gamma\omega)^2}{2^{\tau}}k^{2-\tau-\varsigma\tau}=\frac{(\gamma\omega)^2}{2^\tau} \quad  \text{ for} ~l<2k,
\end{align*}
where $\varsigma$ is taken as $(2-\tau)/\tau$ to guarantee $2-\tau-\varsigma\tau=0$.

$\mathrm{(iv)}$ Let $\max{\{j,k\}}\leq j_\star$. If $j_\star=\frac{2\mathcal{J}_1}{\omega}$,
\begin{align*}
\omega_j\omega_k{>}\frac{(\gamma\omega)^2}{(jk)^{\tau-1}}>\frac{(\gamma\omega)^2}{(j_\star)^{2(\tau-1)}}
=\frac{(\gamma\omega)^2}{(2\mathcal{J}_1/\omega)^{2(\tau-1)}}\stackrel{\omega>\gamma,\gamma\in(0,1),\tau\in(1,2)}{>}\frac{\gamma^4\omega^2}{(2\mathcal{J}_1)^{2(\tau-1)}}.
\end{align*}
On the other hand, for $ j_\star=\left(\frac{6M}{\tilde{M}^2\gamma\omega}\right)^{\frac{1}{1-\varsigma\tau}}$, one has
\begin{align*}
\omega_l\omega_k{>}\frac{(\gamma\omega)^2}{(kl)^{\tau-1}}>\frac{(\gamma\omega)^2}{(k_\star)^{2(\tau-1)}}
=\gamma^2\omega^2\left(\frac{\gamma\omega\tilde{M}^2}{6M}\right)^{\frac{1}{\tau-1}2(\tau-1)}{>}\frac{\gamma^6\omega^2}{(6M)^2/\tilde{M}^4}.
\end{align*}
Since $\omega_{l}=\omega_{-l}$, the remainder of the lemma may be proved in the similar way as  above with $l\geq1,k\leq-1$, or $l\leq-1,k\geq1$, or $l,k\leq-1$. Thus this ends the proof of {\bf(F2)}.
\end{proof}

\section{Appendix}\label{sec:4}
Similar to the proof of Lemmata 2.1--2.3 in \cite{berti2008cantor}, we may get the following  Lemmata \ref{lem19}--\ref{lem21}.
\begin{lemm}[Moser-Nirenberg]\label{lem19}
$\forall u_1,u_2\in \mathcal{H}^{s'}\cap \mathcal{H}^s$ with $s'\geq0$ and $s>\frac{1}{2}$, one has
\begin{align}
\|u_1u_2\|_{{s'}}
&\leq C(s')\left(\|u_1\|_{L^{\infty}(\mathbb{T},H^2(0,\pi))}\|u_2\|_{{s'}}+\|u_1\|_{{s'}}\|u_2\|_{L^{\infty}(\mathbb{T},H^2(0,\pi))}\right)\label{R1}\\
&\leq  C(s')\left(\|u_1\|_{s}\|u_2\|_{{s'}}+\|u_1\|_{{s'}}\|u_2\|_{s}\right).\label{R2}
\end{align}
\end{lemm}
\begin{lemm}[Logarithmic convexity]\label{lem20}
Setting $0\leq\mathfrak{a}'\leq \mathfrak{a}\leq \mathfrak{b}\leq\mathfrak{b}'$ with $\mathfrak{a}+\mathfrak{b}=\mathfrak{a}'+\mathfrak{b}'$, one has
\begin{align*}
\|u_1\|_{\mathfrak{a}}\|u_2\|_{\mathfrak{b}}\leq\Theta\|u_1\|_{\mathfrak{a}'}\|u_2\|_{\mathfrak{b}'}
+(1-\Theta)\|u_2\|_{\mathfrak{a}'}\|u_1\|_{\mathfrak{b}'},\quad \forall u_1,u_2\in \mathcal{H}^{\mathfrak{b}'},
\end{align*}
where $\Theta=\frac{\mathfrak{b}'-\mathfrak{a}}{\mathfrak{b}'-\mathfrak{a}'}$. Particularly, this holds:
\begin{align}\label{R3}
\|u\|_{\mathfrak{a}}\|u\|_{\mathfrak{b}}\leq\|u\|_{\mathfrak{a}'}\|u\|_{\mathfrak{b}'}, \quad \forall u\in \mathcal{H}^{\mathfrak{b}'}.
\end{align}
\end{lemm}
Denote by $\mathscr {C}_k$ the following space composed by the time-independent functions:
\begin{align*}
\mathscr {C}_k:=\left\{f\in C^1([0,\pi]\times\mathbb{R};\mathbb{R}):~u\mapsto f(\cdot,u)\text{ is in }C^{k}(\mathbb{R};H^{2}(0,\pi))\right\}.
\end{align*}
\begin{lemm}\label{lem21}
If $f\in\mathscr {C}_1$, then the composition operator $u(x)\mapsto f(x,u(x))$ belongs to $C(H^2(0,\pi);H^2(0,\pi))$ with
\begin{equation*}
\|f(x,u(x))\|_{H^2}\leq C\left(\max_{u\in[-\mathfrak{U},\mathfrak{U}]}\|f(\cdot,u)\|_{H^2}+\max_{u\in[-\mathfrak{U},\mathfrak{U}]}\|\partial_uf(\cdot,u)\|_{H^2}\|u\|_{H^2}\right),
\end{equation*}
where $\mathfrak{U}:=\|u\|_{L^{\infty}(0,\pi)}$. In particular, one has
\begin{equation*}
\|f(x,0)\|_{H^2}\leq C.
\end{equation*}
\end{lemm}
With the help of Lemmata \ref{lem19}--\ref{lem21}, we can obtain the following lemma.
\begin{lemm}\label{lem22}
Let $f\in\mathcal{C}_{k}$ with $k\geq1$. Then, for all $s>\frac{1}{2},0\leq s'\leq k-1$,  the composition operator $u(t,x)\mapsto f(t,x,u(t,x))$ is in $C(\mathcal{H}^{s}\cap \mathcal{{H}}^{s'};\mathcal{{H}}^{s'})$ with
\begin{align}\label{R4}
\|f(t,x,u)\|_{{s'}}\leq C(s',\|u\|_{s})(1+\|u\|_{{s'}}).
\end{align}
\end{lemm}
\begin{proof}
For all $s'=q\in\mathbb{N}$ with $q\leq k-1$, we show that
\begin{align}\label{R5}
\|f(t,x,u)\|_{q}\leq C(q,\|u\|_{s})(1+\|u\|_{q}),\quad\forall u\in \mathcal{H}^{s}\cap \mathcal{H}^{q},
\end{align}
 and that
 \begin{align}\label{R6}
 f(t,x,u_n)\rightarrow f(t,x,u) \quad\text{as } u_n\rightarrow u ~\text{in}~ \mathcal{H}^{s}\cap \mathcal{H}^{q}.
 \end{align}
 Let us check  \eqref{R5}--\eqref{R6} by a recursive argument. For $q=0~(k=1)$, using \eqref{B2} yields
\begin{align}\label{R7}
\|f(t,x,u)\|_{0}&\leq C\max_{t\in\mathbb{T}}\|f(t,\cdot,u(t,\cdot))\|_{H^2(0,\pi)}
\stackrel{\text{Lemma }\ref{lem21}}{\leq}C'(1+\max_{t\in\mathbb{T}}\|u(t,\cdot)\|_{H^2(0,\pi)})\nonumber\\
&{\leq}C''(1+\|u\|_{s})=:C(\|u\|_{s}).
\end{align}
Moreover, a similar argument as \eqref{R7} can yield that for $k\geq2$,
\begin{align}\label{R8}
\|\partial_{t}f(t,x,u)\|_{0}\leq C(\|u\|_{s}),\quad\max_{t\in\mathbb{T}}\|\partial_{u}f(t,\cdot,u(t,\cdot))\|_{H^2(0,\pi)}\leq C(\|u\|_{s}).
\end{align}
And one has
\[\max_{t\in\mathbb{T}}\|u_{n}(t,\cdot)-u(t,\cdot)\|_{H^2(0,\pi)}\stackrel{\eqref{B2}}{\rightarrow}0 \quad \text { as}~ u_n\rightarrow u ~ \text { in }\mathcal{H}^{s}\cap \mathcal{H}^{0}.
\]
Hence, according to the continuity property in Lemma \ref{lem21} and the compactness of $\mathbb{T}$, we derive
\begin{align*}
\|f(t,x,u_{n})-f(t,x,u)\|_{0}\leq C\max_{t\in\mathbb{T}}\|f(t,\cdot,u_{n}(t,\cdot))-f(t,\cdot,u(t,\cdot))\|_{H^2(0,\pi)}\rightarrow0
\end{align*}
as $ u_n\rightarrow u$ in $ \mathcal{H}^{s}\cap \mathcal{H}^{0}$.

Suppose that \eqref{R5} holds for $q=\mathfrak{k}$ with $\mathfrak{k}\in \mathbb N^+$. Let us show that  it holds for $q=\mathfrak{k}+1$ with $\mathfrak{k}+1 \leq k-1$.

Since $\partial_{t}f,\partial_{u}f\in\mathcal{C}_{k-1}$, above assumption yields that for all $u\in \mathcal{H}^s\cap \mathcal{H}^{\mathfrak{k}+1}$,
\begin{align}\label{R9}
\|\partial_{t}f(t,x,u)\|_{\mathfrak{k}}\leq C(\mathfrak{k},\|u\|_{s})(1+\|u\|_{\mathfrak{k}}),\quad\|\partial_{u}f(t,x,u)\|_{\mathfrak{k}}\leq C(\mathfrak{k},\|u\|_{s})(1+\|u\|_{\mathfrak{k}}).
\end{align}
Setting $\mathfrak{h}(t,x):=f(t,x,u(t,x))$, we write $\mathfrak{h}$ as $\mathfrak{h}(t,x)=\sum_{l\in\mathbb{Z}}\mathfrak{h}_{l}(x)e^{{\rm i}lt}$. Clearly, one has $\mathfrak{h}_{t}(t,x)=\sum_{l\in\mathbb{Z}}{\rm i}l\mathfrak{q}_{l}(x)e^{{\rm i}lt}$. Thus, we have
\begin{align*}
\|\mathfrak{h}(t,x)\|^2_{{\mathfrak{k}+1}}&=\sum\limits_{l\in\mathbb{Z}}(1+l^{2(\mathfrak{k}+1)})\|\mathfrak{h}_{l}\|^2_{H^2}=
\sum_{l\in\mathbb{Z}}\|\mathfrak{h}_{l}\|^2_{H^2}+\sum_{l\in\mathbb{Z}}l^{2\mathfrak{k}}\|{\rm i}l\mathfrak{h}_l\|^2_{H^1}\\
&\leq\|\mathfrak{h}(t,x)\|^2_{0}+\|\mathfrak{h}_{t}(t,x)\|^2_{\mathfrak{k}}
\leq\left(\|\mathfrak{h}(t,x)\|_{0}+\|\mathfrak{h}_{t}(t,x)\|_{\mathfrak{k}}\right)^2.
\end{align*}
This gives rise to
\begin{align}\label{R10}
\|f(t,x,u)\|_{{\mathfrak{k}+1}}\leq \|f(t,x,u)\|_{0}+\|\partial_{t}f(t,x,u)\|_{\mathfrak{k}}+\|\partial_{u}f(t,x,u)\partial_{t}u\|_{\mathfrak{k}},
\end{align}
which leads to that for $q=1$ ($\mathfrak{k}=0$),
\begin{align*}
\|f(t,x,u)\|_{1}&\leq \|f(t,x,u)\|_{0}+\|\partial_{t}f(t,x,u)\|_{0}+C\max_{t\in\mathbb{T}}\|\partial_{u}f(t,\cdot,u(t,\cdot))\|_{H^2(0,\pi)}\|\partial_xu\|_{0}\\
&{\leq}~2C(\|u\|_{s})+C'(\|u\|_{s})\|u\|_{1}\leq C(1,\|u\|_{s})(1+\|u\|_{1})
\end{align*}
because of \eqref{R8}, where $C(1,\|u\|_{s}):=\max\{2C(\|u\|_{s}),C'(\|u\|_{s})\}$. Obviously, one has
\begin{align*}
\begin{cases}
s_1<\mathfrak{k}<s_1+1<\mathfrak{k}+1,\quad \mathfrak{k}=1,\\
s_1<s_1+1<\mathfrak{k}<\mathfrak{k}+1,\quad \forall \mathfrak{k}\geq2,
\end{cases}
\end{align*}
where $s_{1}\in(1/2,\min(1,s))$. Combining this with \eqref{R3}, we can obtain
\begin{align*}
\|u\|_{\mathfrak{k}}\|u\|_{{s_1+1}}{\leq}\|u\|_{{\mathfrak{k}+1}}\|u\|_{{s_1}}\leq\|u\|_{{\mathfrak{k}+1}}\|u\|_{s}.
\end{align*}
Thus it follows from \eqref{R7}--\eqref{R10}, \eqref{R1} that
\begin{align*}
\|f(t,x,u)\|_{{\mathfrak{k}+1}}
\leq & C(\|u\|_{s})+C(\mathfrak{k},\|u\|_{s})(1+\|u\|_{\mathfrak{k}})+C(\mathfrak{k})\|\partial_{u}f(t,x,u)\|_{\mathfrak{k}}\|\partial_{t}u\|_{L^{\infty}(\mathbb{T},H^2(0,\pi))}\\
&+C(\mathfrak{k})
\|\partial_{u}f(t,x,u)\|_{L^{\infty}(\mathbb{T},H^2(0,\pi))}\|u\|_{{\mathfrak{k}+1}}\\
\leq &C(\|u\|_{s})+C(\mathfrak{k},\|u\|_{s})(1+\|u\|_{\mathfrak{k}})+C(\mathfrak{k})C(\mathfrak{k},\|u\|_{s})(1+\|u\|_{\mathfrak{k}})\|u\|_{{s_1+1}}\\
&+C(\mathfrak{k})C(\|u\|_{s})\|u\|_{{\mathfrak{k}+1}}\\
\leq&C(\mathfrak{k}+1,\|u\|_{s})(1+\|u\|_{{\mathfrak{k}+1}}),
\end{align*}
where  $C(\mathfrak{k}+1,\|u\|_{s})=4\max{\left\{C(\|u\|_{s}),C(\mathfrak{k},\|u\|_{s}),C(\mathfrak{k})C(\mathfrak{k},\|u\|_{s})(1+\|u\|_{s}),C(\mathfrak{k})C(\|u\|_{s})\right\}}.$

Finally,  we  assume that \eqref{R6} holds for $q=\mathfrak{k}$. Using the inequality \eqref{R10} yields that \eqref{R6} also holds for  $p=\mathfrak{k}+1$ with $\mathfrak{k}+1\leq k-1$.

When  $s'$ is not an integer, we can obtain the result by the  Fourier dyadic decomposition. The argument is similar to the proof of the Lemma A.1 in \cite{Delort2011}.
\end{proof}
\begin{lemm}\label{lem23}
Let us define a map $F$ as
\begin{align*}
F:\quad \mathcal{H}^{s}\cap \mathcal{H}^{s'}&\rightarrow \mathcal{H}^{s'},\quad u\mapsto f(t,x,u).
\end{align*}
If $f\in\mathcal{C}_{k}$ with $k\geq3$, for all $0\leq s'\leq k-3$,  $F$ is  $C^2$ with respect to $u$ and
\begin{equation*}
{\rm D}_{u}F(u)[h]=\partial_{u}f(t,x,u)h, \quad {\rm D}^2_{u}G(u)[h,h]=\partial^2_{u}f(t,x,u)h^2,\quad \forall h\in \mathcal{H}^s\cap \mathcal{H}^{s'}
\end{equation*}
with
\begin{align}\label{R11}
&\|\partial_{u}f(t,x,u)\|_{{s'}}\leq C(s',\|u\|_{s})(1+\|u\|_{{s'}}),\quad\|\partial^2_{u}f(t,x,u)\|_{{s'}}\leq C(s',\|u\|_{s})(1+\|u\|_{{s'}}).
\end{align}
\end{lemm}
\begin{proof}
Since $\partial_{u}f,\partial^2_{u}f$ are in $\mathcal{C}_{k-1},\mathcal{C}_{k-2}$, Lemma \ref{lem22} shows that the maps $u\mapsto\partial_{u}f(t,x,u)$, $u\mapsto\partial^2_{u}f(t,x,u)$ are continuous and that
formula \eqref{R11} holds. We now verify that $F$  is $C^2$ respect to $u$. Applying the continuity property of $u\mapsto\partial_{u}f(t,x,u)$, we deduce
\begin{align*}
\|f(t,x,u+h)-f(t,x,u)-&\partial_{u}f(t,x,u)h\|_{{s'}}=\|h\int_0^1(\partial_{u}f(t,x,u+\mathfrak{v} h)-\partial_{u}f(t,x,u))\mathrm{d}\mathfrak{v}\|_{{s'}}\\
&\leq C(s')\|h\|_{{\max{\{s,s'\}}}}\max_{\mathfrak{v}\in[0,1]}\|\partial_{u}f(t,x,u+\mathfrak{v} h)-\partial_{u}f(t,x,u)\|_{{\max{\{s,s'\}}}}\\
&=o(\|h\|_{{\max{\{s,s'\}}}}),
\end{align*}
which leads to
\begin{equation*}
{\rm D}_{u}F(u)[h]=\partial_{u}f(t,x,u)h,\quad \forall h\in \mathcal{H}^s\cap \mathcal{H}^{s'}
\end{equation*}
with $u\mapsto {\rm D}_uF(u)$ being continuous. In addition,
\begin{align*}
&\partial_{u}f(t,x,u+\mathfrak{v}  h)h-\partial_{u}f(t,x,u)h-\partial^2_{u}f(t,x,u)h^2=h^2\int_0^1(\partial^2_{u}f(t,x,u+\mathfrak{v} h)-\partial^2_{u}f(t,x,u))~\mathrm{d}\mathfrak{v}.
\end{align*}
The same discussion as above yields that $F$ is twice differentiable with respect to $u$ and that $u\mapsto {\rm D}^2_uF(u)$ is continuous.
\end{proof}

\begin{proof}[\bf{Proof of formula \eqref{G13}}]
If $j>\max\{\mathcal{J}_0,{2\sqrt{M}}\}$, it follows from  formula \eqref{H22} that
\begin{align*}
\inf_{\stackrel{\epsilon\in(\epsilon_1,\epsilon_2)}{w\in\{W\cap \mathcal{H}^s:\|w\|_{s}< r\}}}\left|\mu_{j+1}(\epsilon,w)-\mu_j(\epsilon,w)\right|\geq&1
-\left|\mu_{j+1}(\epsilon,w)-((j+1)^2+\upsilon_0)\right|\\
&-\left|\mu_{j}(\epsilon,w)-(j^2+\upsilon_0)\right|\\
\geq&1-\frac{2M}{j^2}>\frac{1}{2}.
\end{align*}
If $0\leq j\leq\max\{\mathcal{J}_0,{2\sqrt{M}}\}$, it is clear that
\begin{align*}
\mathfrak{s}_j:=\inf_{\stackrel{\epsilon\in(\epsilon_1,\epsilon_2)}{w\in\{W\cap \mathcal{H}^s:\|w\|_{s}< r\}}}\left|\mu_{j+1}(\epsilon,w)-\mu_j(\epsilon,w)\right|.
\end{align*}
Thus we complete the proof.
\end{proof}

\subsection*{Acknowledgement}
The second author would like to thank Prof. Shuguan Ji for many helpful discussions.

\providecommand{\noopsort}[1]{}\providecommand{\singleletter}[1]{#1}%



\end{document}